\documentclass[a4paper,11pt]{amsart}

\usepackage{amsfonts}
\usepackage{amssymb}
\usepackage[utf8]{inputenc}
\usepackage{amsmath}
\usepackage{pdflscape}
\usepackage{graphicx}
\usepackage{comment}
\usepackage[colorlinks=true, linkcolor=red, citecolor=blue]{hyperref}
\usepackage[]{epsfig}
\usepackage[]{pstricks}
\usepackage{tikz}

\makeatletter
\@namedef{subjclassname@2020}{\textup{2020} Mathematics Subject Classification}
\makeatother

\numberwithin{equation}{section}

\newtheorem{theorem}{Theorem}[section]
\newtheorem{proposition}[theorem]{Proposition}
\newtheorem{lemma}[theorem]{Lemma}
\newtheorem{corollary}[theorem]{Corollary}

\newtheorem{claim}[theorem]{Claim}
\newtheorem{remark}{Remark}
\theoremstyle{remark}

\newtheorem{definition}{Definition}[section]

\newcommand{\R}{\mathbb{R}}

\newcommand{\N}{\mathbb{N}}

\DeclareMathOperator{\supp}{supp}
\numberwithin{equation}{section}

\begin{document}
\title [ Exponential Stabilization of BBM-KP]{ Global   Stabilization for the BBM-KP equations on $\R^2$}
\author[Gallego]{F. A. Gallego}
\address{Departamento de Matem\'atica y Estadística, Universidad Nacional de Colombia (UNAL), Cra 27 No. 64-60, 170003, Manizales, Colombia}
\email{fagallegor@unal.edu.co}
\author[Gonzalez Martinez]{V. H. Gonzalez Martinez}
\address{Departamento de Matem\'atica, Universidade Federal de Pernambuco, S/N Cidade Universit\'aria, 50740-545, Recife (PE), Brazil}
\email{victor.martinez@ufpe.br}
\author[Mu\~noz Grajales]{J. C. Muñoz Grajales}
\address{Departamento de Matem\'aticas, Universidad del Valle, Calle 13, No 100-00, Ciudad Universitaria Mel\'endez, Cali, Colombia}
\email{juan.munoz@correounivalle.edu.co}

\begin{abstract}
In this paper, we present results on the energy decay of the BBM-KP equations (I and II) posed on $\R^2$ with localized damping. This model offers an alternative to the KP equations, analogous to how the regularized long-wave equation relates to the classical Korteweg–de Vries (KdV) equation. We show that the energy associated with the Cauchy problem decays exponentially when a localized dissipative mechanism is present in a subdomain. Finally, we validate the theoretical results on the exponential stabilization of solutions to the BBM-KP equations with damping through numerical experiments using a spectral-finite difference scheme.
\end{abstract}

\maketitle

\section{Introduction}

\subsection{The model}
The main purpose of this paper is to study the exponential decay
of energy for the initial value problem associated \textcolor{black}{with} the Benjamin-Bona-Mahony Kadomtsev-Petviashvili equations (BBM-KP) which is the regularized version of the usual KP equation:
\begin{equation}\label{1}
		u_t+ u_x + u^pu_x -u_{xxt}+ \gamma \partial_x^{-1} u_{yy} =0,
\end{equation}
where $\gamma=\pm 1$ and $p\geq 1$ is an integer.  This model arises in various contexts where nonlinear dispersive waves propagate principally along the $x$-axis with weak dispersive effects undergone in the direction parallel to the $y$-axis and normal to the main direction of propagation. Similarly to the KP equations, the case $\gamma = 1$ corresponds to the BBM-KP-II equation, while $\gamma = -1$ leads to the BBM-KP-I model. It is \textcolor{black}{known} that the flows of the KP-I and KP-II equations, considered in natural spaces, behave very
differently, namely, the KP-II equation has a semilinear character, while the KP-I equation behaves as a quasilinear equation.

The BBM-KP equation has been studied from various mathematical points of view. For example, Bona et al. \cite{bona2002} investigated several classes of evolution equations (RLW-KP-type), including natural generalizations of a regularized version of the KP equation for the propagation of long surface water waves. They studied global well-posedness in appropriate function spaces and developed a theory of non-linear stability for certain traveling waves in cases where the equation admits solitary-wave solutions. Later, Saut et al. \cite{saut2004}, using the Fourier transform restriction method, demonstrated the existence of unique global-in-time solutions in a suitable functional space continuously embedded in $C(\mathbb{R}, H^1_x(\mathbb{R}^2))$. Furthermore, they established the orbital stability of localized lump-like solutions for the equation with $\gamma=-1$, improving the results presented in \cite{bona2002}. \textcolor{black}{In the context of stabilization properties, there are several works related to BBM-type equations. In particular, we can cite the work on the global uniform stabilization of the BBM-KdV equation presented in \cite{rosier}, where the author shows that stabilizing the BBM-KdV equation reduces to proving a Unique Continuation Property, which is established by developing a Carleman inequality. It is important to note that the BBM-KP equation can be seen as a two-dimensional generalization of the BBM-KdV equation.}

It is interesting to point out that there are many works related to equation \eqref{1} not only dealing with well-posedness theory. For example, in \cite{mammeri} the author proves that if the solution of this problem has compact support for any time, then the solution vanishes identically. This is called a unique continuation property. Moreover, a variety of exact traveling wave solutions of the BBM-KP equation and its generalizations have been formally derived, including solitons, compactons, solitary patterns, and periodic solutions, among others (see, for instance, \cite{aguilar2024, bona2002, mammeri, saut2004}). Furthermore, \cite{daspan2007} shows that the solutions of the IVP (KP) and regularized KP equations (BBM-KP) are the same, within their precision, on the time scale \(0 \leq t \leq \varepsilon^{-3/2}\), where nonlinear and dispersive effects can create a one-order difference in the wave profiles.  Finally, and very recently, in \cite{aguilar2024}, the authors show that the solution of the BBM-KP Cauchy problem converges to the solution of the BBM Cauchy problem in a suitable function space if the initial data for both equations are as close as \( y \rightarrow \infty \).

\vglue 0.3cm

\subsection{Setting of the Problem}
Our main purpose is to address two mathematical issues connected to  global well-posedness and large-time behavior of solutions of the BBM-KP with the \textcolor{black}{presence} of a damping term
\begin{equation}\label{e1}
\begin{cases} 
		u_t+ u_x + u^pu_x -u_{xxt}+ \gamma \partial_x^{-1}
    u_{yy} + a(x,y)u=0,& (x,y)\in \R^2,\ t>0. \\
	u(x,y,0)=u_0(x,y), & (x,y)\in \R^2,
\end{cases} 
\end{equation}
where $\gamma=\pm$ and $p\geq 1$ is  an integer. The primary objective of this paper is to demonstrate the exponential decay of the energy associated with equation \eqref{e1} for cases \( p = 1 \) and \( p = 3 \) with a ``localized" damping term. This type of dissipation was initially proposed in \cite{menzala2002} to control the KdV equation on a bounded interval and was later extended to models involving unbounded wave propagation domains (see, for example, \cite{cavalcanti2012, linares2009} and the references therein). More precisely, we present
new contributions with respect to the decay of the energy related to mild solutions for the damped BBM-KP equations \eqref{e1}, where \textcolor{black}{the} the non-negative function $a(x,y)$ is responsible for the dissipative effect and we establish the exponential decay of solutions of the energy and the well-posedness in a Bourgain Spaces continuously embedded in $C(\R;H^1_x(\R^2))$.

\textcolor{black}{First}, note that equation \eqref{e1} is a generalization of the well-known BBM-KP  \eqref{1}. In the case of the initial value problem related to \eqref{1}, as point out in \cite{saut2004}, the value 
\begin{equation}\label{energy}
    E(t)=\frac{1}{2}\|u(\cdot,\cdot,t)\|_{H^1_x}^2
\end{equation} 
can be interpreted as the energy. It is obvious that for a smooth and decaying at infinity solution $u(x,,y t)$ to \eqref{1}, the energy is a constant of motion, that is
\begin{align*}
    \|u(\cdot,\cdot,t)\|_{H^1_x}^2=\|u_0(\cdot,\cdot)\|_{H^1_x}^2
\end{align*}
Hence, the energy does not decay as \( t \rightarrow \infty \). Therefore, the primary objective is to establish the decay of energy in the cases where \( a(x,y) \neq 0 \) without assuming the existence of a positive constant \(\alpha_0\) such that \( a(x,y) \geq \alpha_0 \). To our knowledge, there are no studies that address the exponential stabilization of equations \eqref{e1} in the presence of localized damping.

Now, considering the energy $E(t)$ associated to the the model \eqref{e1}, at least formally, the solutions of \eqref{e1} should satisfy
\begin{equation}\label{diss-energy}
\displaystyle \frac{d}{dt}E(t) =  - \int_{\mathbb{\R}^2} a(x,y) u^2  dxdy,
\end{equation}
for any positive $t$. Then, if we assume that $a(x,y)\geq \alpha_0$, for some $\alpha_0 >0$, it is forward to infer that $E(t)$ converges to zero exponentially. In contrast, when the damping function $a$ is allowed to change sign or is effective only on the \textcolor{black}{in a} subset of the domain, the problem is much more subtle. Moreover,  whether \eqref{diss-energy}  generates a flow that can be continued indefinitely in the temporal variable, defining a solution valid for all $t\geq 0$, is a non-trivial question. In this paper, we consider a localized damping which acts only \textcolor{black}{in a} bounded subset of the line, more precisely,
\begin{equation}\label{hyp a}
\begin{cases}
\text{$a \in H^s(\R^2)$, \textcolor{black}{$s>\frac{3}{2}$} is nonnegative and  $a(x,y)\geq \lambda_0 >0$} \\
\text{ almost everywhere in $\R^2\setminus \omega$ where $\omega:=(-B, B) \times (C,D) $, }\\
\text{for some $B>0$  and $C<D$ such that $D-C < \pi$.}
\end{cases}
\end{equation}
Thus, given the damping described by \eqref{hyp a}, a natural question arises:
\vspace{0.2cm}
\begin{center}
\textit{Does $ E(t) \rightarrow 0$ as $ t \rightarrow \infty$? If so, can we give the decay rate?}
\end{center}
Our work focused on a specific choice of damping effect and aims to establish that this model predicts interesting qualitative properties observed in certain initial value problems of nonlinear dispersive equations. The problem of the time decay rate for this type of model posed in an unbounded domain has received considerable attention in recent years; see, for instance, \cite{alam2013, capistrano2024, cavalcanti2012, cavalcanti2014, doronin2016, faminskii2021, gallego2019, komornik2020, linares2009, pazoto2010, saut1993_2, saut2001, song2010, wang2021, wang2023, wazwaz2005, wazwaz2008} and references therein.

\subsection{Preliminar and notations}

This Cauchy problem is dealt with in the articles of Bona, Liu and Tom \cite{bona2002},
or Saut and Tzvetkov \cite{saut2004}. They proved that for all initial datum in the
subspace of $L^2(\R^2)$ provided with the norm $(\|u\|_2^2 + \|u_x\|_2^2)^{\frac12}$, there exists
a unique global in time solution. Indeed, If we
denote by $X^s(\R^2)$ the space of functions given by
\begin{equation}
X^s(\R^2):=\left\lbrace f \in H^{s}(\R^2): \partial_x^{-1}f_y\in H^{s-1}(\R^2)\right\rbrace
\end{equation}
 provided with the norm
\begin{equation}
\|f\|_{X^s}^2:=\|f\|_{H^s}^2 + \| \partial_x^{-1} f_y\|_{H^{s-1}}^2.
\end{equation}
Here and below, $\partial_x^{-1}f_y$ is defined via Fourier transform as
\begin{equation}
\widehat{\partial_x^{-1}f_y}(\xi,\eta)=\frac{\eta}{\xi}\widehat{f}(\xi,\eta).
\end{equation}
Note that if $s\geq 1$; then $\partial_x^{-1}f_y \in L^2(\R^2)$,  so there is a $g \in L^2(\R^2)$ such that
$f_y= g_x$ at least in the sense of distribution. On the other hand, since $f \in H^s(\R^2)\subset H^1(\R^2)$; so $f_y \in  L^2(\R^2)$; whence $g_x \in L^2(\R^2)$. Thus $g$ lies in the Hilbert space $H^1_x(\R^2)$; where
\begin{equation}
H^{\lambda}_x(\R^2)=\{f\in L^2(\R^2): (1+\xi^2)^{\lambda/2}\widehat{f}(\xi,\eta) \in L^2(\R^2) \}
\end{equation}
and similarly, we can define 
\begin{equation}
H^{\lambda}_y(\R^2)=\{f\in L^2(\R^2): (1+\eta^2)^{\lambda/2}\widehat{f}(\xi,\eta) \in L^2(\R^2) \}
\end{equation}
supplied with the obvious norms. Moreover, for $\beta \geqslant 0$, we consider
$$
V_\beta\left(\mathbb{R}^2\right)=\left\{f \in L_2\left(\mathbb{R}^2\right):|\xi|^\beta \hat{f}(\xi, \eta), \frac{\eta}{\xi} \hat{f}(\xi, \eta) \in L_2\left(\mathbb{R}^2\right)\right\}
$$
with the norm
$$
\|f\|_{V_\beta\left(\mathbb{R}^2\right)}=\left(\int_{\mathbb{R}^2}\left(1+|\xi|^{2 \beta}+\frac{\eta^2}{\xi^2}\right)|\hat{f}(\xi, \eta)|^2 d \xi d \eta\right)^{1 / 2} .
$$
On the other hand, if $\beta>0$, let
\begin{equation}
W_{\beta}(\R^2)=\{f\in H^{\beta+1}_x(\R^2): \partial_x^{-1}f_y \in H^{\beta}_x(\R^2)\}
\end{equation}
with norm, 
\begin{equation}
\|f\|_{W_{\beta}(\R^2)}=\|f\|_{H^{\beta}_x(\R^2)}+\|\partial_x^{-1}f_y\|_{H^{\alpha/2}_x(\R^2)}
\end{equation}
and let $\widetilde{W}_{\beta}(\R^2)$ be the space
\begin{equation}
\widetilde{W}_{\beta}(\R^2)=\{f\in H^{\beta+1}_x(\R^2): \partial_x^{-1}f_y \in H^{\beta}_x(\R^2)\}
\end{equation}
with norm, 
\begin{equation}
\|f\|_{\widetilde{W}_{\beta}(\R^2)}=\|f\|_{H^{\beta}_x(\R^2)}+\|\partial_x^{-1}f_y\|_{H^{\beta}_x(\R^2)}.
\end{equation}
Note that if $\beta > 1/2$, then for any $f \in W_{\beta}(\R^2)$, 
\begin{equation}\label{e4}
\int_{-\infty}^{\infty}f_y(x,y)dx=0,
\end{equation}
where for almost every $y$; the left-hand side of \eqref{e4} is an improper Riemann integral. As above when delineating $\partial_x^{-1}f_y$, define $\partial_x^{-2}f_{yy}$ via the relation
\begin{equation}
\widehat{\partial_x^{-2}f_{yy}}(\xi,\eta)=\frac{\eta^2}{\xi^2}\widehat{f}(\xi,\eta).
\end{equation}
If this quantity lies in $L^2(\R^2)$; then $f_{yy}= h_{xx}$ for some $h \in L^2(\R^2)$; and thus as
in \eqref{e4}, if $\partial_x^{-2}f_{yy} \in H^{\beta}_x (\R^2)$, for some $\beta >  1/2$; then
\begin{equation}
\int_{-\infty}^{\infty}\int_{-\infty}^{x}f_{yy}(s,y)dsdx=0,
\end{equation}

\subsection{Main Result}
The next result asserts that \((1.1)\) is globally uniformly exponentially stable in
\(H^{1}_x(\mathbb{R}^2)\). This means that the decay rate \(\nu\) in Proposition \ref{th:MUes} below is independent of \(R\) when \(\color{red} \left\|u_{0}\right\|_{H^{1}_x} \leqslant R\). More precisely, we have the following definition
 \begin{definition}
System \eqref{e1} is said to be {\em locally uniformly exponentially stable}
in $H^1_x(\R^2)$ if for any $r > 0$ there exist two constants $C > 0$ and $\eta  > 0$ such that for any
$u_0 \in  H^1_x(\R^2)$ with $\|u_0\|_{H^1_x} < r$ and for any   solution $u = u(x,y, t)$ in $X^{s}$ of \eqref{e1}, it holds that
\begin{equation}\label{localexponentialdefinition}
\|u(t)\|_{H^1_x}^2 \leq C\|u_0\|_{H^1_x}^2e^{-\eta t}, \quad \forall t \geq 0 
\end{equation}
If the constant $\eta $ in \eqref{localexponentialdefinition} is independent of $r$, the system \eqref{e1} is said to be {\em globally
uniformly exponentially} stable in $H^1_x(\R^2)$.
\end{definition}

\begin{theorem}[Global Uniform Stabilization]\label{globalexp}
Let $s>3/2$, $a\in L^\infty(\Omega)$ satisfy Assumption \ref{hyp a} and $\gamma=\pm 1$ and $p=1$. Then, \eqref{e1} is globally uniformly exponentially stable in \(H^{1}_x(\mathbb{R}^2)\),
i.e. there exists a positive constant \(\nu\) that does not depend on initial data such that, for any $\varepsilon>0$ and \( u_0 \in X^s \cap W_1(\R^2)\),  the corresponding solution $u$ of \eqref{e1} belongs to the space $C(\R,X^s)$ such that
\begin{equation}\label{exponentialdecay}
    E(t)\leqslant \alpha\left( E(0)\right) \mathrm{e}^{-\nu t}, \quad t \geqslant \varepsilon,
\end{equation}
where $E$ is the energy functional given by \eqref{energy}  and \(\alpha_{\varepsilon}: \mathbb{R} \rightarrow\)
\(\mathbb{R}\) is a  non-negative continuous function.
\end{theorem}

One of the main tools employed in this paper is the Compactness-Uniqueness Argument. In the context of stabilization, it was initially used in \cite{littman}, crediting the idea to L. Hörmander. However, some attribute the method to Lions in \cite{lions}. However, this approach simplifies our task of proving an \textit{ observability inequality} for the non-linear system \eqref{e1}, which is reduced to having some unique continuation property for the model.

\subsection{Paper’s outline}
The paper is organized as follows:
\vglue 0.4cm

In Section \ref{Wellposednessand UCP}, we analyze the global well-posedness of the system \eqref{e1} when the initial data belongs to the space $X^s \cap W_1(\R^2)$, where $s > \frac{3}{2}$, and we establish the unique continuation property for the system \eqref{e1} in the absence of a damping term. Section \ref{stabilizationsection} is devoted to the exponential stabilization of the system \eqref{e1} in $H^1_x(\mathbb{R}^2)$, that is, to the proof of Theorem \ref{globalexp}. In Section \ref{simulations}, we check the performance of our feedback laws through some numerical simulations. Finally, in Section \ref{furthercomments}, we conclude the work with some further considerations.

\section{Global Well-posedness and Unique Continuation Property}\label{Wellposednessand UCP}

The primary objective of this section is to establish the global well-posedness of the BBM-KP equation \eqref{e1}, incorporating a localized damping term \(a(\cdot, \cdot)u\), where \(a\) meets the condition \eqref{hyp a}. To achieve this, we adapt the approach used in \cite{bona2002}. Specifically, we consider
 the solution of the initial value problem \eqref{e1}
formally as
\begin{equation}\label{duhamel}
u(x,y,t)=K_tu_0 - \int_0^t K_{t-\tau}Q\left( \frac{u^{2}(\cdot, \tau)}{2}\right)d\tau  - \int_0^t K_{t-\tau}Q\left(a(\cdot)u(\cdot,\tau)\right)d\tau,
\end{equation}
where the operators $K_t$ and $Q$ are defined via their Fourier transforms,
\begin{equation}
\widehat{K_t f}(\xi,\eta)=e^{\frac{-i(\xi^2 + \gamma \eta^2)t}{\xi(1+|\xi|^2)}}\widehat{f}(\xi,\eta)
\end{equation}
and
\begin{equation}
\widehat{Qf}(\xi,\eta)=\frac{i\xi}{1+|\xi|^2}\widehat{f}(\xi,\eta)
\end{equation}
\textcolor{black}{Before presenting} the well-posedness result, we establish some auxiliary results.  The following results show some properties of the operators $K_t$ and $Q$. 

\begin{proposition}
The operator $K_t$ is a unitary operator on all the spaces $L^2\left(\mathbb{R}^2\right)$, $H^s\left(\mathbb{R}^2\right)$, $X_s$, $H_x^\lambda\left(\mathbb{R}^2\right)$, $V_2\left(\mathbb{R}^2\right)$, $W_2\left(\mathbb{R}^2\right)$ and $\tilde{W}_2\left(\mathbb{R}^2\right)$. Moreover, $Q$ is a bounded linear operator from $L^2\left(\mathbb{R}^2\right)$ into $H_x^{1}\left(\mathbb{R}^2\right)$, from $H_y^1\left(\mathbb{R}^2\right)$ into $V_{1}\left(\mathbb{R}^2\right)$, from $H^1\left(\mathbb{R}^2\right)$ into $W_{1}\left(\mathbb{R}^2\right)$, from $H_y^1\left(\mathbb{R}^2\right)$ into $\tilde{W}_{1}\left(\mathbb{R}^2\right)$, and from $H^s\left(\mathbb{R}^2\right)$ into $X_s$, for any $s \geqslant 1$.
\end{proposition}
\begin{proof}
    The proposition follows from \cite[Proposition 3.1]{bona2002}  with $\alpha=2$.
\end{proof}

\begin{corollary}\label{coro1}
\begin{enumerate}
\item  The composite operator $K_t Q$ is a bounded linear operator from $L^2\left(\mathbb{R}^2\right)$ to $H_x^{1}\left(\mathbb{R}^2\right)$ for all $\alpha \geqslant 1$.
\item $K_t Q$ is bounded from $H_y^1\left(\mathbb{R}^2\right)$ to $V_{1}\left(\mathbb{R}^2\right)$.
\item $K_t Q$ maps continuously $H^1\left(\mathbb{R}^2\right)$ and $H_y^1\left(\mathbb{R}^2\right)$ to $W_{1}\left(\mathbb{R}^2\right)$ and $\tilde{W}_{1}\left(\mathbb{R}^2\right)$, respectively. 
\item $K_t Q$ is bounded from $H^s\left(\mathbb{R}^2\right)$ to $X_s$ for all $s \geqslant 1$.
\end{enumerate}
\end{corollary}
The upcoming embedding result will be useful in demonstrating the local well-posedness of the initial value problem.
 
 \begin{lemma}
 For $\beta>\frac{1}{2}, W_\beta\left(\mathbb{R}^2\right)$ being continuously embedded in $C_{\mathrm{b}}\left(\mathbb{R}^2\right)$, the bounded continuous functions defined in $\mathbb{R}^2$.
 \end{lemma}
\begin{proof}
     The proposition follows from \cite[Lemma 3.3]{bona2002}.
\end{proof}

\subsection{Global Well-posedness}
 First, we will establish the local well-posedness of the solutions for the BBM-KP equation with a damping term. This involves developing a local existence theory under very weak assumptions.
\begin{theorem}\label{Localwellposedness}
Let $a \in H^s(\R^2)$  and let $u_0 \in X_s$ with $s>\frac{3}{2}$. Then there exists $T>0$ such that the initial value problem \eqref{e1} has a unique solution
$$
u \in C(0, T ; X_s) \cap C^1(0, T ; H^{s-2}(\mathbb{R}^2)).
$$
In particular, the solution $u$ also satisfies
$$
u \in C(0 ; T ; H^s(\mathbb{R}^2)) \quad \text { and } \quad \partial_x^{-1} u_y \in C(0, T ; H^{s-1}(\mathbb{R}^2)).
$$
The solution depends continuously in these function classes on variations of $u_0$ in $X_s$.\end{theorem}

\begin{proof} 
The strategy is to first show that corresponding to given $\phi \in X_s$, the integral equation \eqref{e1} has a solution in $C([0, T] ; X_s)$ for suitable values of $T>0$. This will be accomplished via the contraction-mapping principle. For given $\phi$ in $X_s$ and any $v \in C(0, T ; X_s)$, define the action of the operator $A=A_\phi$ on $v$ to be
\begin{equation*}
    A_{\phi}(v)(t)=K_t\phi(x,y) - \int_0^t K_{t-\tau}Q\left( \frac{v^{2}(\cdot, \tau)}{2}\right)(x,y)d\tau - \int_0^t K_{t-\tau}Q\left(a(\cdot)v(\cdot,\tau)\right)(x,y)d\tau,
\end{equation*}
for $(x, y, t) \in \mathbb{R}^2 \times[0, T]$. The aim is just as above, to show that the operator $A$ is a contraction of the closed ball 
$$B_R(0)=\{u \in C(0, T ; X_s): \sup_{t \in [0,T]} \|u(t)\|_{X_s}\leq R\}$$
provided $R$ and $T$ are well chosen. The crux of the matter is to understand the temporal integral in the definition of $A_\phi$. First, let us show that $A_\phi(B[0,R]) \subset B[0,R]$ for $R>0$ to be chosen later. From the continuity of the operator $K_tQ:H^{s}(\mathbb{R}^{2})\rightarrow X_s$ and taking into account that $H^s$ is a Banach-algebra for $s>\frac{3}{2}$, we obtain
\begin{equation*}
    \begin{aligned}
        \|A_\phi(v)(t)\|_{X_s}\leq& \|K_t\phi(\cdot)\|_{X_s}+\int_{0}^{t}\|K_{t-\tau}Q \left(\frac{v^{p+1}(\cdot,\tau)}{p+1}\right)\|_{X_s} d\tau+\int_{0}^{t}\|K_{t-\tau}Q(a(\cdot)u(\cdot,\tau))\|_{X_s}d\tau\\
        \leq& \|\phi(\cdot)\|_{X_s}+C\int_{0}^{t}\|\left(\frac{v^{p+1}(\cdot,\tau)}{p+1}\right)\|_{H^s} d\tau+C\int_{0}^{t}\|a(\cdot)v(\cdot,\tau)\|_{H^s}d\tau\\
        \leq& \|\phi(\cdot)\|_{X_s}+\frac{C}{p+1}\int_{0}^{t}\|v(\cdot,\tau)\|_{H^s}^{p+1} d\tau+C\|a(\cdot)\|_{H^s}\int_{0}^{t}\|v(\cdot,\tau)\|_{H^s}d\tau\\
        \leq&\|\phi(\cdot)\|_{X_s}+\frac{C}{p+1}\|v\|_{C(0,T;X_s)}^{p+1}T+\|a(\cdot)\|_{H^{s}}\|v\|_{C(0,T;X_s)}T\\
        \leq&\|\phi(\cdot)\|_{X_s}+\frac{C}{p+1}R^{p+1}T+\|a(\cdot)\|_{H^{s}}RT\\
    \end{aligned}
\end{equation*}
So, taking $R=3\|\phi(\cdot)\|_{X_s}$ and $T>0$ satisfying 
\begin{equation*}
    T<\frac{p+1}{3CR^{p}} \hbox{ and } T<\frac{1}{3\|a(\cdot)\|_{H^{s}}}
\end{equation*}
we obtain
\begin{equation*}
        \|A_\phi(v)\|_{C(0,T;X_s)}\leq R .
\end{equation*}
On the other hand,
\begin{equation*}
    \begin{aligned}
\left\|\int_0^t K_{t-\tau} Q\left(\frac{1}{p+1} u^{p+1}\right) d \tau\right. & \left.-\int_0^t K_{t-\tau} Q\left(\frac{1}{p+1} v^{p+1}\right) d \tau\right\|_{X_s} \\
& \leqslant \int_{0}^{T}\left\|K_{t-\tau} Q\left(\frac{1}{p+1}\left(u^{p+1}-v^{p+1}\right)\right)\right\|_{X_s}d\tau \\
& \leqslant C\int_{0}^{T}\left\|\frac{1}{p+1}\left(u^{p+1}-v^{p+1}\right)\right\|_{H^s}d\tau \\
& \leqslant \frac{C}{p+1}\int_{0}^{T}\left\|(u-v)\sum_{i=0}^{p} u^{p-i}v^{i}\right\|_{H^s}d\tau\\
& \leqslant \frac{C}{p+1}\int_{0}^{T}\left(\sum_{i=0}^{p}R^{p-i}R^i\right) \|u-v\|_{H^s}d\tau\\
& \leqslant \frac{C}{p+1} T \left(\sum_{i=0}^{p}R^{p-i}R^i\right) \|u-v\|_{C(0,T;X_s)}.
\end{aligned}
\end{equation*}
and
\begin{equation*}
    \begin{aligned}
\left\|\int_0^t K_{t-\tau} Q\left(a(\cdot)(u(\cdot,\tau)-v(\cdot,\tau)\right) d \tau\right\|_{X_s}  &\leqslant  \int_0^t \left\|K_{t-\tau} Q\left(a(\cdot)(u(\cdot,\tau)-v(\cdot,\tau)\right)\right\|_{X_s} d \tau\\
& \leqslant  C\int_0^t \left\|a(\cdot)(u(\cdot,\tau)-v(\cdot,\tau))\right\|_{H^s} d \tau\\
& \leqslant  C\|a(\cdot)\|_{H^s}\int_0^t \left\|u(\cdot,\tau)-v(\cdot,\tau)\right\|_{H^s} d \tau\\
& \leqslant   CT\|a(\cdot)\|_{H^s}\|u-v\|_{C(0,T;X_s)}.
\end{aligned}
\end{equation*}
Combining the previous estimates, we obtain
\begin{equation*}
    \|A_\phi(u)-A_\phi(v)\|_{C(0,T;X_s)}\leq T\left(\frac{C}{p+1} \left(\sum_{i=0}^{p}R^{p-i}R^i\right)+C\|a(\cdot)\|_{H^s}\right) \|u-v\|_{C(0,T;X_s)}
\end{equation*}
Therefore, taking $T>0$ satisfying 
\begin{equation*}
    T< \frac{p+1}{4C \left(\sum_{i=0}^{p}R^{p-i}R^i\right)} \hbox{ and } T< \frac{1}{4C\|a(\cdot)\|_{H^s}},
\end{equation*}
it follows that $A_\phi$ is a contraction. Consequently, we can apply the Banach fixed point theorem to conclude that the map $A_\phi$ has a unique fixed point.
\end{proof}

To establish global well-posedness, note that the solution formula \eqref{duhamel} implies that the solution of the BBM-KP equation \eqref{e1} can be written in the form:
$$
u(x,y,t) = u_1(x,y,t) +   u_2(x,y,t) 
$$
where $u_1$ is solution of 
\begin{equation}
\begin{cases} 
		u_t+ u_x + u^pu_x -u_{xxt}+ \gamma \partial_x^{-1}
    u_{yy} =0,& (x,y)\in \R^2,\ t>0. \\
	u(x,y,0)=u_0(x,y), & (x,y)\in \R^2,
\end{cases} 
\end{equation}
and $u_2$ is solution of 
\begin{equation}
\begin{cases} 
		u_t+ u_x  -u_{xxt}+ \gamma \partial_x^{-1}
    u_{yy} + a(x,y)u=0,& (x,y)\in \R^2,\ t>0. \\
	u(x,y,0)=0, & (x,y)\in \R^2,
\end{cases} 
\end{equation}
with 
$$
u_2(x,y,t)=- \int_0^t K_{t-\tau}Q\left(a(\cdot)u(\cdot,\tau)\right)d\tau.
$$

\begin{theorem}\label{Globalwellposedness}
Let  $p \leq 2$ $s>\frac{3}{2}$, $a \in H^s(\R^2)$  and $u_0 \in X_s \cap W_{1}\left(\mathbb{R}^2\right)$   be such that $\partial_x^{-2} u_{0, y y} \in L^2\left(\mathbb{R}^2\right)$. Then for any $T>0$, the initial value problem \eqref{e1} has a unique solution
$$
u \in C\left(0, T ; X_s\right) \cap C^1\left(0, T ; H^{s-2}\left(\mathbb{R}^2\right)\right)
$$
In particular, for any $T>0$, the solution $u$ of \eqref{e1} has the properties
$$
u \in C\left(0, T ; H^s\left(\mathbb{R}^2\right)\right), \quad \partial_x^{-1} u_y \in C\left(0, T ; H^{s-1}\left(\mathbb{R}^2\right)\right)
$$
For any $T>0$, the solution map $u_0 \mapsto u$ is locally Lipschitz-continuous from $X_s \cap W_{1}\left(\mathbb{R}^2\right)$  into these spaces.
\end{theorem}
\begin{proof}
From \cite[Theorem 4.2]{bona2002}, it follows that 
$$ \sup_{t\in[0,T]}\|u_1(\cdot,t)\|_{X^s} \leq C\|u_0\|_{X^s} $$
and Corollary \ref{coro1}, we deduce that 
$$ \sup_{t\in[0,T]}\|u_2(\cdot,t)\|_{X^s} \leq  \int_0^t \|K_{t-\tau}Q(a(\cdot)u(\cdot,\tau))\|_{X^s}d\tau \leq C \int_0^t \|a(\cdot)u(\cdot,\tau))\|_{H^s}d\tau\leq C\|u_0\|_{X^s} $$
By using the fact that $H^s(\R^2)$ is a Banach-algebra for $s > \frac{3}{2}$, we deduce that
$$ \sup_{t\in[0,T]}\|u_2(\cdot,t)\|_{X^s} \leq C T\|a\|_{H^s}\sup_{t\in[0,T]}\|u(\cdot,t)\|_{X^s}.$$
Then, by from Theorem \ref{Localwellposedness}, it yields that
\begin{equation}\label{e160}
    \sup_{t\in[0,T]}\|u(\cdot,t)\|_{X^s} \leq C \left( 1 + T\|a\|_{H^s} \right)\|u_0\|_{X^s}.
\end{equation}
\textcolor{black}{Then, the \textit{a priori}} estimation \eqref{e160} implies the global well-posedness.
\end{proof}

Let $E(t)$ be the energy defined in \eqref{energy}. Next, we prove that the energy $E(t)$ is a decreasing function of $t$.
 \begin{proposition}
Let  $s>\frac{3}{2}$, $p \leq 2$ , $a \in H^s(\R^2)$ satisfying \eqref{hyp a}  and $u_0 \in X_s \cap W_{1}\left(\mathbb{R}^2\right)$   be such that $\partial_x^{-2} u_{0, y y} \in L^2\left(\mathbb{R}^2\right)$. Consider the energy $E(t)$ given by \eqref{energy}. Then, we have
\begin{equation}\label{diss-energy_1}
\displaystyle \frac{d}{dt}E(t) =  - \int_{\mathbb{\R}^2} a(x,y) u^2  dxdy,
\end{equation}
 \end{proposition}
 \begin{proof}
 Consider $f_x(x,y)= u_y(x,y)$, hence we get 
 \begin{align*}
 \frac{d}{dt}E(t) &=\int_{\R^2}u(t)u_t(t) dxdy + \int_{\R^2}u_x(t)u_{xt}(t) dxdy \\
 &= \int_{\R^2}u(t) \Big (u_t(t) dxdy -u_{xxt}(t)\Big) dxdy \\
 &= \int_{\R^2}u(t) \Big (-u(t)u_x(t)- \gamma\partial_x^{-1} u_{yy}(t)-a(x,y)u(t)\Big ) dxdy \\
  &=-\frac{1}{3} \int_{\R^2}(u^3(t))_xdxdy- \gamma\int_{\R^2}u(t)f_{y}(t)dydx-\int_{\R^2}a(x,y)u^2(t) dxdy \\
    &= \gamma\int_{\R^2}u_y(t)f(t)dydx-\int_{\R^2}a(x,y)u^2(t) dxdy     = \frac12 \gamma\int_{\R^2}(f^2(t))_xdxdy-\int_{\R^2}a(x,y)u^2(t) dxdy \\
 &=  -\int_{\R^2}a(x,y)u^2(t) dxdy 
 \end{align*}
 \end{proof}
\color{black}

\subsection{Unique Continuation Property} \label{secUCP}

In this section, we establish the Unique Continuation Property (UCP) for the distributional solutions of the BBM-KP equations. It is important to note that in order to achieve an exponential decay rate for our problem, we must reduce the stabilization problem to a UCP for a distributional solution of \eqref{e1} without the damping term present. In \cite{mammeri}, the authors established the UCP for solutions in \(C([-T,T]; \hat{X}^s(\mathbb{R}^2))\), where \(\hat{X}^s\) denotes the space of functions \(f\) in \(H^s(\mathbb{R}^2)\) such that \(\partial_x^{-1} f\) also belongs to \(H^s(\mathbb{R}^2)\), equipped with the norm
\[
|f|_s = \left(\|f\|_s^2 + \left\|\partial_x^{-1} f\right\|_s^2\right)^{1 / 2}.
\]
We need to make some modifications to adapt their approach, since our aim is to prove the UCP for distributional solutions supported in a compact set. This approach was introduced in \cite{constantin}, which is based on integration by parts.

Before presenting the proof of UCP for a compact support distributional solution, we need to establish the following previous results.
\vglue 0.4cm

Let $I$ be an open interval of $\mathbb{R}$ and for all $t \in I$, $T_t$ be an element of $\mathcal{D}'(\Omega)$. 

\begin{definition}
    Let $k \in \mathbb{N} \cup \{0\}$. We say that $(T_t) \in C^k(I, \mathcal{D}'(\Omega))$ if for all $\varphi \in C_0^{\infty}(\Omega)$ the map from $I$ to $\mathbb{C}$, $t \mapsto \langle T_t,\varphi \rangle$ is of class $C^k(I)$.
\end{definition}

\begin{proposition}\label{Zuily}
    Let $(T_t) \in C^k(I,\mathcal{D}'(\Omega))$. For all $0 \leq \ell \leq k$ and for all $t \in I$, there exists a distribution $T_t^{\ell}$ such that $(T_t^{\ell}) \in C^{k-\ell}(I,\mathcal{D}'(\Omega))$ and
    \begin{equation}
        \left[\left( \frac{d}{dt}\right)^{\ell}\langle T_t,\varphi \rangle \right](t_0)=\langle T_{t_0}^{(\ell)},\varphi \rangle, \quad \forall t_0 \in I, \forall \varphi \in C_{0}^{\infty}(\Omega).
    \end{equation}
\end{proposition}
 \begin{proof}
    See [\cite{Zuily}, Proposition 3.2., page 61].
 \end{proof}

\begin{lemma}
    Let $u$ be a distributional solution of the Cauchy problem associated with the BBM-KP-II equation \eqref{e1} with $a(\cdot,\cdot)=0$, $\gamma=1$ and $p= 1$. If, for all $t \in [-T,T]$, $u(t)$ has compact support, then for all $t \in [-T,T]$ and $g \in C^\infty(\mathbb{R})$,
    \begin{equation*}
        \langle e^{x}g(y)(u-u_{xx})(t),1 \rangle_{\mathcal{E}'(\mathbb{R}^{2}),\mathcal{E}(\mathbb{R}^{2})}=0,
    \end{equation*}
where $\mathcal{E}'(\mathbb{R}^{2})=(C^\infty(\mathbb{R}^{2})'$ denotes the space of the distributions with compact support.
\end{lemma}
\begin{proof}
    Let $G \in C^\infty(\mathbb{R}^{2},\mathbb{R})$ and $t \in [-T,T]$. Since $u(t)$ has compact support, from the definition of derivatives in the sense of distributions, we have
    \begin{equation*}
        \langle (u- u_{xx})(t),G(x,y)\rangle_{\mathcal{E}'(\mathbb{R}^{2}),\mathcal{E}(\mathbb{R}^{2})}=0,
    \end{equation*}
    if
    \begin{equation*}
        G(x,y)- G_{xx}(x,y)=0.
    \end{equation*}
    So, we set $G(x,y)=e^{x}g(y)$.
\end{proof}

\begin{theorem}\label{UCP}
Let $u$ be a distributional solution of the Cauchy problem associated with the BBM-KP-II equation \eqref{e1} with $a(\cdot,\cdot)=0$, $\gamma=1$ and $p=1$. Let us suppose that there exist $0 < B < +\infty$ and $D > C$ in $\R$, with $D-C<\pi$, such that for all $t \in [-T, T]$,
\begin{equation*}
\sup u(t) \subset [-B,B]\times [C,D].
\end{equation*}
Then $u$ vanishes identically.
\end{theorem}

\begin{proof}
     Thanks to Proposition \ref{Zuily} we have,
\begin{equation*}
        \langle e^{x}g(y)(u_t-u_{xxt})(t),1 \rangle_{\mathcal{E}'(\mathbb{R}^{2}),\mathcal{E}(\mathbb{R}^{2})}=0,
\end{equation*}
and since $u$ is a distributional solution of equation \eqref{e1}, we get
\begin{equation}
    0=\langle e^{x}g(y)(u_{x}+\alpha u^pu_{x}+ \partial_{x}^{-1}u_{yy})(t),1\rangle_{\mathcal{E}'(\mathbb{R}^{2}),\mathcal{E}(\mathbb{R}^{2})}.
\end{equation}
From  \cite[Lemma 2.1]{mammeri}, the derivative definition in the sense of distributions in the $x$-direction and taking into account that $u$ is essentially bounded, we have
\begin{equation*}
    0=\left\langle e^{x}g(y)\left(u+\frac{\alpha}{p+1}u^{p+1}\right)(t),1\right\rangle_{\mathcal{E}'(\mathbb{R}^{2}),\mathcal{E}(\mathbb{R}^{2})}+ \left\langle e^{x} g(y)u_{yy}(t),1 \right\rangle_{\mathcal{E}'(\mathbb{R}^{2}),\mathcal{E}(\mathbb{R}^{2})}.
\end{equation*}
Taking two derivatives in the $y$-direction, it follows that
\begin{equation*}
    0=\left\langle e^{x} \left( g(y) + g''(y) \right) u(t),1\right\rangle_{\mathcal{E}'(\mathbb{R}^{2}),\mathcal{E}(\mathbb{R}^{2})}+ \frac{\alpha}{p+1}\left\langle e^{x}g(y)u^{p+1}(t),1 \right\rangle_{\mathcal{E}'(\mathbb{R}^{2}),\mathcal{E}(\mathbb{R}^{2})}.
\end{equation*}
Our aim is to solve the ordinary differential equation
\begin{equation*}
    g(y)+  g''(y)=0 \hbox{ with } g(y)>0 
\end{equation*}
One solution is given by $g(y)=\cos \left(y+y_0\right)$, with $y_0$ to be chosen. Finally,  since $p=1$, we find
$$\int_{\mathbb{R}^{2}}e^{x}\cos \left(y+y_0\right) u^{p+1}(x,y,t) dxdy=0 \Rightarrow \int_C^D\int_{-B}^{B} e^{x}\cos \left(y+y_0\right) u^{2}(x,y,t) dxdy=0 .$$
Since $D-C < \pi$, there exists $y_0$ belongs to $(-C-\pi/2, -D + \pi/2)$ and this implies that $e^{x}\cos \left(y+y_0\right)>0$ on the support of $u(t)$.Then, $u$ vanishes identically.
\end{proof}

\begin{corollary}\label{UCP2}
Let $u$ be a distributional solution of the Cauchy problem associated with the BBM-KP-I equation \eqref{e1} with $a(\cdot,\cdot)=0$, $\gamma=-1$ and $p=1$. Let us suppose that there exist $0 < B < +\infty$ such that for all $t \in [-T, T]$,
\begin{equation}
\sup u(t) \subset [-B,B]\times [-B,B].
\end{equation}
Then $u$ vanishes identically.
\end{corollary}
\begin{proof}
    The proof is similar to the KP-BBM-II proof by taking the function $G(x,y)=e^{x+y}$.
\end{proof}

\begin{remark}\label{remark1}
    The unique continuation property remains valid if we consider the case $p = 3$, or more generally any odd integer $p$. However, a significant challenge arises with respect to the well-posedness of the problem. In this context, we only have established global solutions for cases $p = 1$ and $p = 2$. For higher odd values of $p$, particularly beyond $p = 2$, the existence and uniqueness of global solutions become much more complex and are not as well understood. This difficulty stems from the lack of regularity and control in the higher-order nonlinear terms, which complicates the analysis and often leads to issues such as blow-up or nonexistence of solutions under standard assumptions. Consequently, the study of well-posedness for these higher-order cases remains an open problem, requiring more refined techniques and potentially novel approaches to address the challenges posed by these nonlinearities.
\end{remark}

\section{Global Uniform Stabilization}\label{stabilizationsection}

In this section, we study the long-term behavior of solutions to \eqref{e1}. Our goal is to demonstrate that the energy associated with the BBM-KP equations decays exponentially. First, let us prove Proposition \ref{th:MUes} concerning the local exponential stability of the system \eqref{e1}.

\begin{proposition}[Local Uniform Stabilization]\label{th:MUes}
Let $s > 3/2$, $R>0$ and $a\in H^{s}(\R^2)$ satisfy Assumption \ref{hyp a}. Suppose that $u$ is the solution of system \eqref{e1} with initial data $u_0 \in X^s \cap W_1(\R^2)$ with $$ | u_0 | _Xs \cap W_1\leq R$$ with $p=1$. Then, there exists $C=C(R)>0$ and $\nu=\nu(R)>0$ such that the energy $E(t)$, defined in \eqref{energy} decays exponentially as $t$ tends to infinity, i.e., 
\begin{equation}
E(u(t)) \leq C e^{-\nu t} E(u_0), \quad \forall t \geq 0.
\end{equation}
\end{proposition}			

First, note that the corresponding solution $u$ of (\ref{e1}) satisfies the following estimate
\begin{equation}\label{e152}
\|u(t)\|_{H^1_x(\R^2)}^{\color{red}2} +2\int_0^t\int_{\R^2}a(x,y)|u(x,y, \tau)|^2dxdyd\tau = \|u_0\|_{H^1_x(\R^2)}^2.
\end{equation}
On the other hand, multiplying the equation in (\ref{e1}) by $(T-t)u$ and integrating on $\R^2\times [0,T]$, we obtain
\begin{equation}\label{e153}
T \|u_0\|_{H^1_x(\R^2)}^2 = \frac{1}{2}\|u\|_{L^2(0,T;H^1_x(\R^2))}^2+\int_0^T\int_{\R^2}(T-t)a(x,y)|u(x,y,t)|^2dxdydt.
\end{equation}
\begin{claim}
For any $T > 0$ and $R> 0$ there exist $C = C(R, T )$ such that the following
 estimate holds for any weak solution $u$ of \eqref{e1} with $\|u_0\|_{X^s \cap W_1(\R^2)}\leq R$ such that
\begin{equation}\label{observabiblity_1}\color{black}
\|u_0\|_{H^1_x(\R^2)}^2  \leq C \int_0^T \int_{\R^2}a(x,y)|u(x,y,t)|^2dxdydt
\end{equation}
\end{claim}
\begin{proof}
In order to prove the observability inequality \eqref{observabiblity_1}, \textcolor{black}{we will use} some compactness-uniqueness argument, it means that we argue by contradiction to suppose that (\ref{observabiblity_1}) does not hold. Hence, there exists a sequence  of initial data $\{u_{0,n}\}$ belong to $X^s \cap W_1(\R^2)\subset X^s \subset H^1_x(\R^2), s>3/2$, where the corresponding solutions $\{u_n\}$ satisfies 
\begin{equation}\label{boundedu0}
\|u_{0,n}\|_{H^1_x}^2\leq   \|u_{0,n}\|_{X^s} \leq  R
\end{equation}
and
\begin{equation}\label{e157_1}\color{black}
\lim_{n \rightarrow \infty}\frac{\int_0^T\int_{\R^2}a(x,y)|u_n|^2dxdydt}{ \|u_{0,n}\|_{H^1_x(\R^2)}^2}=0.
\end{equation}
Let $\{\lambda_n\}$  and $\{v_n\}$  be sequences defined respectively by
\begin{equation*}\color{black}
\alpha_n =\|u_{0,n}\|_{H^1_x(\R^2)} \quad \text{and} \quad v_n(x,y,t) = \frac{1}{\alpha_n} u_{n}(x,y,t).
\end{equation*}
Thus, the sequence $\{\alpha_n\}$ is bounded, there exists a subsequence (denoting by $n$) such that $\alpha_n \rightarrow \alpha$, noting that $\alpha_n>0$ for all $n$, by \eqref{e157_1}. Moreover,  
\begin{equation}\label{enorm1_1}\color{black}
\|v_{0,n}\|_{H^1_x}=1, \quad \forall  n\in \N.
\end{equation}
We observe that, for all $n \in  \N$ the function $v_n$ is the mild solution of the following problem 
\begin{equation*}\label{e159}
\begin{cases}
\partial_t v_n + \partial_x v_n -\partial_t\partial_x^2v_n+\alpha_n v_n\partial_xv_n+\textcolor{red}{\gamma} \partial_x^{-1}\partial_y^2 v_n+a(x,y)v_n =0 \\
v_n(x,0)=\frac{u_n(x,0)}{\alpha_n}.
\end{cases}
\end{equation*}
for $(x,y) \in \R^2$. Moreover, from \eqref{e157_1}, it follows that 
\begin{align}\label{e158}
  \lim_{n \rightarrow \infty}\int_0^T\int_{\R^2}a(x,y)|v_n|^2dxdydt =0.
\end{align} 
\textcolor{black}{Clearly, we have that $\{v_n(\cdot,\cdot,0)\}_{n\in \N}$ is bounded in $H^1_x(\R^2)$.}   Combining (\ref{e158}) and (\ref{e152}) we conclude that $\{v_n\}$ is bounded in $L^{\infty}(0,T;H^1_x(\R^2))$. \textcolor{black}{On the other hand}, from the well-posedness Theorem \ref{Globalwellposedness} we already know that the solution depends continuously in these function classes on variations of initial data in $X^s \cap W_1(\R^2)$, thus we have that 
 \begin{equation}\label{limitacionesbuenas}
 \begin{cases}
      \|v_n\|_{L^{\infty}(0,T;X^s)} \leq C \|v_n(0)\|_{X^s\cap W_1(\R^2)} \\
      \|v_{n,t}\|_{L^{\infty}(0,T;H^{s-2})} \leq C \|v_n(0)\|_{X^s\cap W_1(\R^2}  \\
       \|\partial_x^{-1} \partial_y v_{n}\|_{L^{\infty}(0,T;H^{s-1})} \leq C \|v_n(0)\|_{X^s\cap W_1(\R^2} 
 \end{cases}
     \end{equation}
 \textcolor{black}{Thus, we deduce that }
 \begin{align}
    &\text{ $\{v_n\}$ is bounded in  $L^{\infty}(0,T;X^s)$ }\\
    &\text{ $\{v_{n_t}\}$ is bounded in  $L^{\infty}(0,T;H^{s-2}(\R^2))$ } \\
    &\text{ $\{\partial_{x}^{-1}\partial_y^2 v_{n}\}$ is bounded in  $L^{\infty}(0,T;H^{s-2}(\R^2))$ }.
 \end{align}

 Let $K$ be a compact subset of $\mathbb{R}^{2}$ and let's denote by $X^{s}_{K}$ the subspace of $X^{s}$ consisting of the functions with support inside $K$. We will use the symbol $\overset{c}{\hookrightarrow}$ to denote a compact embedding. From the Rellich's Theorem and taking into account that $s>3/2$ we have that $$X^s_{K} \hookrightarrow  H^{s}_{K} \overset{c}{\hookrightarrow} H^{1} \hookrightarrow H^{1}_x \hookrightarrow H^{s-2},$$
 so we can apply  \cite[Corollary 4, page 85]{simon1986compact} to conclude that $\{v_n\}$ is relative compact in $L^{2}(0,T;H^1_x(K))$. Consequently, there exists a subsequence, still denoted by  $\{v_n\}$, such that
\begin{equation}\label{e50}
\text{$v_{n} \longrightarrow v$ in $L^{2}(0,T;H^1_{x}(K))$ strongly for any compact $K\subset \R^2$ and a.e.}
\end{equation} 
 Now, extracting a subsequence if needed, we have
\begin{align}
&v_{n} \rightarrow v  \quad \text{in $L^{2}([0,T];L^2_{loc}(\R^2))$ strongly}, \quad \label{limites_1} \\
&v_{n,x} \rightarrow v_x  \quad \text{in $L^{2}([0,T];L^2_{loc}(\R^2))$ strongly}, \label{limites_2} \\
&\partial_x^{-1}\partial_{y}^2  v_n \rightharpoonup f  \quad \text{in $L^{2}([0,T];H^{s-2}(\R^2))$ weak},  \label{limites_3}
\end{align}
as $n \rightarrow \infty$. Furthermore, we affirm that $f= \partial_x^{-1}\partial_y^2 v.$ In fact, consider $\partial_x^{-1}\partial_y^2 v_n=g_n$, thus $\partial_y^2 v_n=\partial_xg_n$ and from \eqref{limites_3}, we have that $\partial_y^2 v_n \rightarrow \partial_y^2 v$ in $\mathcal{D}'$ and $\partial_xg_n = \partial_x\partial_x^{-1}\partial_y^2 v_n \rightarrow \partial_x f $ in $\mathcal{D}'$, which implies that $\partial_x f =  \partial_y^2 v$. Moreover, note that \eqref{e50} implies that limit in \eqref{e158} satisfies 
\begin{align}\label{suportofv}
 \int_0^T\int_{\R^2}a(x,y)|v|^2dxdydt =0 \quad \Rightarrow \quad \text{$v \equiv 0$ on $[0,T]\times \R^2\setminus\omega$}
\end{align}
\textcolor{black}{and from \eqref{enorm1_1}, we obtain 
\begin{equation}\label{limitnorm1}
    \|v_0\|_{H^1_x(\R^2)}=1.
\end{equation}}
Thus from \eqref{limites_1}  and \eqref{limites_2}, we deduce that 
\begin{equation}\label{limitecase1}
    v_n v_{n,x} \rightarrow vv_x  \quad \text{in $L^{2}([0,T];L^2_{loc}(\R^2))$ strongly}, 
\end{equation}
Let us consider $\varphi$ be an arbitrary function belongs to $C^{\infty}([0,T]\times \R^2)$. Thus,  for any $k$ and considering the distributional derivative, it follows that
$$\color{black}
\int_0^T \int_{\mathbb{R}^2}\left(v_k \varphi_t-v_k \varphi_{xxt}- v_k\varphi_x+\alpha_kv_k v_{k,x} \varphi+\gamma\partial_x^{-1}v_{k,yy} \varphi+a v_k \varphi\right) d x d t=0, \quad \gamma=\pm 1.
$$
Passing to the limit as $k \rightarrow+\infty$ and taking \eqref{limites_1}, \eqref{limites_2}, \eqref{limites_3} \eqref{suportofv} and \eqref{limitecase1}, into account, we obtain the following equality
$$\color{black}
\int_0^T \int_{\mathbb{R}^2}\left(v \varphi_t-v \varphi_{xxt}-v \varphi_x+\alpha v v_{x} \varphi+\gamma\partial_x^{-1}v_{yy} \varphi\right) d x d t=0,  \quad \gamma=\pm 1
$$
Thus, $v$ solves the equation
\begin{equation}\label{solutionwhole_1}
\begin{cases}\color{black}
v_t+v_x-v_{xxt}+\alpha vv_x+\gamma\partial_x^{-1}v_{yy}=0 \quad & \color{black}\text{in $\mathcal{E}'([0,T]\times \R^2)$},\\
\color{black} \supp v(t) \subset \omega & \color{black} \hbox{for all } t \in [0,T].
\end{cases}
\end{equation}
 Therefore, by the unique continuation property established in
Theorem \ref{UCP}, we conclude that \(v \equiv 0\) in \(\R^2 \times(0, T)\) which contradicts \eqref{limitnorm1}. Thus \eqref{observabiblity_1} holds.
\end{proof}
\color{black}
\begin{proof}[\textbf{Proof of Proposition \ref{th:MUes}}]
Note that (\ref{e152}) implies that 
\begin{equation}
\|u(t)\|_{H^1_x(\R^2)}^2  - \|u_0\|_{H^1_x(\R^2)}^2=-2\int_0^t\int_{\R^2}a(x,y)|u(x,y, \tau)|^2dxdyd\tau \quad 0\leq t\leq T. 
\end{equation}
From observability inequality \eqref{observabiblity_1}, it follows that 
\begin{equation*}
-2\int_{0}^{T} \int_{\mathbb{R}^{2}}a(x)|u(x,y,t)|^2dxdydt \leq -\frac{2}{C}\|u_0\|_{H^1_x(\R^2)}^2, \quad C=C(R,T).
\end{equation*}
Since the energy is dissipative, it follows that
\begin{equation*}
\|u(T)\|_{H^1_x(\R^2)}^2  - \|u_0\|_{H^1_x(\R^2)}^2 \leq -\frac{2}{C}\|u(T)\|_{H_{x}^{1}(\mathbb{R}^{2})}^2,
\end{equation*}
so we have
\begin{equation*}
\|u(T)\|_{H^1_x(\R^2)}^2 \leq \delta \|u_0\|_{H^1_x(\R^2)}^2, \quad \text{with $0< \delta <1$ and $\delta=\delta(R,T)=\frac{C}{C+2}$} 
\end{equation*}
Consequently,
\begin{equation*}
\|u(kT)\|_{H^1_x(\R^2)}^2 \leq \delta^k \|u_0\|_{H^1_x(\R^2)}^2, \quad \forall k\geq 0.
\end{equation*}\color{black}
Finally, fixing $T>0$, for any $t \geq 0$, there exists $k>0$, such that $kT \leq t < (k+1)T$. Thus,
\begin{align*}
\|u(t)\| _{H^1_x(\R^2)}^2 &\leq \|u(kT)\|_{H^1_x(\R^2)}^2 \leq \delta^k \|u_0\|_{H^1_x(\R^2)}^2 \\
& \leq \delta^{\frac{t}{T}}\delta^{-1} \|u_0\|_{H^1_x(\R^2)}^2 \\
&\leq \delta^{-1}\|u_0\|_{H^1_x(\R^2)}^2e^{-\nu t},
\end{align*}
where $\nu=\nu(R,T) = -\frac{\ln \delta}{T} > 0$.
\end{proof}

\color{black}

The next result asserts that the system (\ref{e1}) is globally uniformly exponentially stable in $H^1_x(\R^2)$. This means that the constant $\nu$  in Proposition \ref{th:MUes} is independent of $R$, when $\|u_0\|_{X^s} \leq R$.

\begin{proof}[\textbf{Proof of Theorem \ref{globalexp}}]
The theorem is a direct consequence of Proposition \ref{th:MUes}, as the decay \(\nu\) can be taken
as the decay for \(R=1\) (the decay rate is given by the behavior of the solutions in
a neighborhood of the origin, since all trajectories enter into this neighborhood). In fact, by Proposition \ref{th:MUes}, there exist $\nu '= \nu'(\varepsilon)>0$ and $C=C(\varepsilon)>0,$ such that
\begin{equation*}
\|u(t)\|_{H^1_x(\R^2)} \leq C\|u(\varepsilon)\|_{H^1_x(\R^2)}e^{-\nu' t}, \quad t\geq \varepsilon.
\end{equation*}
If $\|u_0\|_{X^s}\leq R$, again, by Proposition \ref{th:MUes} there exist $C_R >0$ and $\nu_R >0$, satisfying
\begin{equation*}
\|u(t)\|_{H^1_x(\R^2)} \leq C_R\|u_0\|_{H^1_x(\R^2)}e^{-\nu_R t}, \quad t\geq 0.
\end{equation*}
Thus, for all $t\geq \varepsilon$, we have
\begin{align*}
\|u(t)\|_{H^1_x(\R^2)} &\leq C\|u(\varepsilon)\|_{H^1_x(\R^2)}e^{-\nu' t} \\
& \leq CC_{\textcolor{black}{R}}\|u_0\|_{H^1_x(\R^2)}e^{-\nu_R \varepsilon}e^{-\nu' t} \\
& \leq \alpha_\varepsilon(\|u_0\|_{H^1_x(\R^2)})e^{-\nu' t},
\end{align*}
where $\alpha_{\varepsilon}(s)=CC_{\textcolor{black}{R}}e^{-\nu_R \varepsilon} s$.
\end{proof}

\section{Numerical approximation of the BBM-KP equation}\label{simulations}

In this section, we introduce a numerical scheme to approximate solutions of the BBM-KP equation
\begin{equation}\label{BBMKP2}
u_t+ u_x + \alpha u^p u_x -u_{xxt}+ \gamma \partial_x^{-1} u_{yy} + a(x,y)u=0,
\end{equation}
considered in this paper with $\alpha=1$. The proposed scheme combines a spectral approach for the spatial variables with a second-order two-step finite difference method for the temporal variable. The numerical solver is validated using the exact solution in the linear regime and some periodic traveling wave solutions in the non-linear case. Although our implementation is specifically demonstrated for $p=1$, the numerical solver can be adapted to approximate the solutions of the BBM-KP equation for any positive integer $p$.

\subsection{The numerical scheme}

To obtain numerical solutions for equation \eqref{BBMKP2}, the unbounded domain $\mathbb{R}^2$ is approximated by a rectangle $[-L_x/2, L_x/2] \times [-L_y/2, L_y/2]$ with sufficiently large dimensions $L_x$ and $L_y$ to maintain the periodicity assumptions. This rectangular domain is then discretized into $N_x \times N_y$ equidistant points, where $N_x$ and $N_y$ are even integers.

Consider a solution $u(.,., t) \in L_{per}^2([-L_x/2, L_x/2] \times [-L_y/2, L_y/2] )$ to equation \eqref{BBMKP2} with $p =1$. For a fixed time $t$, the 2-dimensional Fourier transform of $u(.,., t)$ with respect to the spatial variables $x,y$, is defined as
\[
\hat{u}(m, n, t) := \frac{1}{L_x L_y} \int_{-L_y/2}^{L_y/2} \int_{-L_x/2}^{L_x/2} u(x,y,t) e^{- 2 \pi i( \frac{m x}{L_x} + \frac{n y}{L_y} ) }dx dy, ~~ (m, n) \in \mathbb{Z}^2.
\]
Therefore, 
\begin{equation}\label{uapprox}
u(x, y, t) =  \sum_{m \in \mathbb{Z} } \sum_{n \in \mathbb{Z} } \hat{u}(m, n, t) e^{ 2\pi i ( \frac{m x}{L_x} + \frac{n y}{L_y} )}, ~~ (x, y) \in \mathbb{R}^2.
\end{equation}
Let $\xi_m = \frac{2 \pi m}{L_x}$, $\eta_n = \frac{2 \pi n}{L_y}$. By substituting the expression for $u$ given by \eqref{uapprox} into the linear terms of the BBM-KP equation \eqref{BBMKP2} (i.e. excluding the nonlinear term $u^2 /2$), and then performing inner product on the space $L_{per}^2([-L_x/2, L_x/2] \times [-L_y/2, L_y/2] )$ with the periodic orthogonal basis functions, $\phi_{m,n}(x,y) = e^{ 2\pi i ( \frac{m x}{L_x} + \frac{n y}{L_y} ) }$, where $m,n \in \mathbb{Z}$,
we find that the Fourier coefficient of $\hat{u}(m,n)$ satisfy
\[
i \xi_m \hat{u}_t - (i \xi_m)^3 \hat{u}_t + \alpha (i \xi_m)^2 \widehat{ \frac{u^2}{2} } + (i \xi_m)^2 \hat{u} + \gamma (i \eta_n)^2 \hat{u} + i \xi_m \widehat{a u} = 0, ~~  m, n \in \mathbb{Z},
\]
o equivalently,
\begin{align}
   \hat{u}_t = -i \frac{ \xi_m^2 + \gamma \eta_n^2}{ \xi_m (1 + \xi_m^2)} \hat{u} -i \frac{ \alpha \xi_m \widehat{ \frac{u^2}{2}} - i \widehat{a u} }{1 + \xi_m^2}, ~~ m, n \in \mathbb{Z}.
\end{align}
Equation above is approximated by the second-order two-step numerical finite difference scheme
\begin{align}\label{two_step_scheme}
   \hat{u}^{(k+1)} = H_1(\Delta t, \xi_m, \eta_n) \hat{u}^{(k-1)} + W_1(\Delta t, \xi_m, \eta_n) \Big(  \alpha \xi_m \widehat{ \frac{ (u^{(k)})^2 }{2} } - i \widehat{a u^{(k)}} \Big), ~~~ k = 0,1,2,...,
\end{align}
where $\Delta t$ is the time step and
\[
    H_1(\Delta t, \xi_m, \eta_n) := \frac{1 - \frac{i \Delta t( \xi_m^2 + \gamma \eta_n^2)}{\xi_m (1 + \xi_m^2)}  }{1 +\frac{i \Delta t( \xi_m^2 + \gamma \eta_n^2)}{\xi_m (1 + \xi_m^2)}},
\]
and
\[
   W_1(\Delta t, \xi_m, \eta_n) := - \frac{ \frac{2 i \Delta t}{1 + \xi_m^2} }{1 +\frac{i \Delta t( \xi_m^2 + \gamma \eta_n^2)}{\xi_m (1 + \xi_m^2)}  }.
\]
Here the notation $u^{(k)}$ means the approximation of the function $u$ at time $t = k \Delta t$.

The two-step scheme in \eqref{two_step_scheme} can be initialized, for example, by employing the one-step scheme
\begin{align}
   \hat{u}^{(k+1)} = H_2(\Delta t, \xi_m, \eta_n ) \hat{u}^{(k)} + W_2(\Delta t, \xi_m, \eta_n ) \Big(  \xi_m \widehat{ \frac{ (u^{(k)})^2 }{2} } - i \widehat{a u^{(k)}} \Big), ~~~ k=0,1,....,
\end{align}
with
\[
    H_2(\Delta t, \xi_m, \eta_n) := \frac{1 - \frac{i \Delta t( \xi_m^2 + \gamma \eta_n^2)}{2 \xi_m (1 + \xi_m^2)}  }{1 +\frac{i \Delta t( \xi_m^2 + \gamma \eta_n^2)}{2 \xi_m (1 + \xi_m^2)}},
\]
and
\[
   W_2(\Delta t, \xi_m, \eta_n ) := - \frac{ \frac{ i \Delta t}{1 + \xi_m^2} }{1 +\frac{i \Delta t( \xi_m^2 + \gamma \eta_n^2)}{2 \xi_m (1 + \xi_m^2)}  }.
\]
In the previous numerical schemes, the Fourier coefficients of the solution $u$ are approximated using the Fast Fourier Transform (FFT) algorithm, applied to the equidistant mesh in the rectangle $[-L_x/2, L_x/2] \times [-L_y/2, L_y/2]$. The series in equation \eqref{uapprox} is then \textcolor{black}{approximated by the truncated} series
\begin{equation}\label{uapprox2}
u(x, y, t) \approx  \sum_{m = - \frac{N_x}{2} }^{ \frac{N_x}{2} } \sum_{n= - \frac{Ny}{2} }^{ \frac{Ny}{2} } \hat{u}(m, n, t) e^{ 2\pi i ( \frac{m x}{L_x} + \frac{n y}{L_y} )}, ~~ (x, y) \in \mathbb{R}^2.
\end{equation}
To avoid \textcolor{black}{aliasing} from being applied in the calculation of the product terms $\widehat{\frac{u^2}{2}}$ and $\widehat{a u}$, we use the so-called 3/2th rule. This method involves extending $N$ Fourier coefficients by adding $\frac{3}{2} N$ zeros (a technique known as zero padding), \textcolor{black}{and then selecting the relevant anti-aliasing} Fourier coefficients to ensure an accurate resolution of these product terms.

\subsection{Some exact solutions of the BBM-KP equation}

In this section, we present some exact solutions to the BBM-KP equation \eqref{e1} \textcolor{black}{that are useful to test} the proposed numerical solver.

Consider the initial value problem associated to the linearized BBM-KP equation (i.e. equation \eqref{BBMKP2} with $\alpha = 0$):
\begin{align}\label{linearBBMKP}
\begin{cases}
&u_t+ u_x - u_{xxt} + \gamma \partial_x^{-1} u_{yy} + a u=0, (x,y)\in \R^2,\ t>0, \\
& u(x,y,0) = u_0(x,y),
\end{cases}
\end{align}
where $a$ is a real constant. By using the 2d-Fourier transform, we 
obtain the exact solution of the problem above given by
\begin{equation}\label{exact_sol}
   u(x,y,t) =  \sum_{m \in \mathbb{Z} } \sum_{n \in \mathbb{Z} } \hat{u}_0(m, n) e^{-i t \frac{ \xi_m^2 + \gamma \eta_n^2 - i a \xi_m}{ \xi_m (1 + \xi_m^2 ) }   } e^{ 2\pi i ( \frac{m x}{L_x} + \frac{n y}{L_y} )}, ~~ (x, y) \in \mathbb{R}^2.
\end{equation}
Observe that the exact solution \eqref{exact_sol} includes a time-dependent factor $e^{-at /(1 + \xi_m^2) }$, which decays exponentially as $t \to \infty$ for $a>0$, \textcolor{black}{according to} the theory presented in the previous sections. Consequently, the coefficient $a$ effectively \textcolor{black}{promotes} has a damping effect on the solutions of the linear BBM-KP equation \eqref{linearBBMKP}.

However, in the absence of a damping effect (i.e. $a=0$),
periodic traveling wave solutions (i.e. solutions with a permanent shape) to the full BBM-KP equation \eqref{BBMKP2} exist. One family of such traveling wave solutions, with speed $c$, is given by
\begin{equation}\label{periodic_wave_BBMKP}
    u(x,y,t) = r_3 -(r_3 -r_2) \text{sn}^2 \Big( \sqrt{ - \frac{ \tilde{\alpha}(r_3 - r1)}{6 c} } (x + ry - ct ); m \Big),
\end{equation}
with $\tilde{\alpha} = -\frac12 \alpha$, which was derived in \cite{ouyang2014} under the assumptions $c/\tilde{\alpha}<0$, $(1 + \gamma r^2 + c)/c < 0$. Furthermore, $r_1< r_2 < r_3$ are the three real zeros of the polynomial
\[
p(\phi) = \frac{2 \tilde{\alpha} }{3 c} \phi^3 - \frac{1 + \gamma r^2 - c}{c}\phi^2 + 2 h,
\]
with the parameter $h$ selected such that $h_1 < h < h_2 = 0$, where
\[
h_1 = \frac{(1 + \gamma r^2 -c)^3}{6 \tilde{\alpha}^2 c} < 0.
\]
Furthermore, $m = (r_3 - r_2)/(r_3 - r_1)$. Here, $\text{sn}(l; m)$ denotes the Jacobian elliptic function $\text{sn}$ with modulus $m$.

\subsection{Numerical experiments}

In this section, we illustrate the evolution of the exact solutions to the BBM-KP equation \eqref{e1} presented in the previous section, using the numerical solver described in \eqref{two_step_scheme} for some values of the modeling parameters. For all cases where it is not explicitly stated otherwise, we use $\alpha = 1$.

First, Figures \ref{lineargammapositive} and \ref{lineargammanegative} present the simulation results for two different values of the parameter $\gamma$:  $\gamma=1$ and $\gamma=-1$, respectively. The initial condition used for these simulations is given by the Gaussian function
\begin{equation}\label{Gaussian_pulse}
u(x,y,0)= u_0(x,y) = 0.5 e^{-\sigma (x^2 + y^2) },
\end{equation}
with $\sigma =4$. For both values of $\gamma$, the cross-sectional profile $u(x,0, t=2)$ is displayed at time $t=2$. In these simulations, damping is applied with a constant damping value $a = 1$. We observe a good agreement between exact and numerical solutions in both scenarios, indicating the accuracy and reliability of the numerical method employed. The numerical parameters used are $\Delta t = 0.001$, $\Delta x = \Delta y = 600/2^{12} \approx 0.1465$.

\begin{figure}[ht]
\centering
 \includegraphics[width=0.7\textwidth, height=0.4\textwidth ]{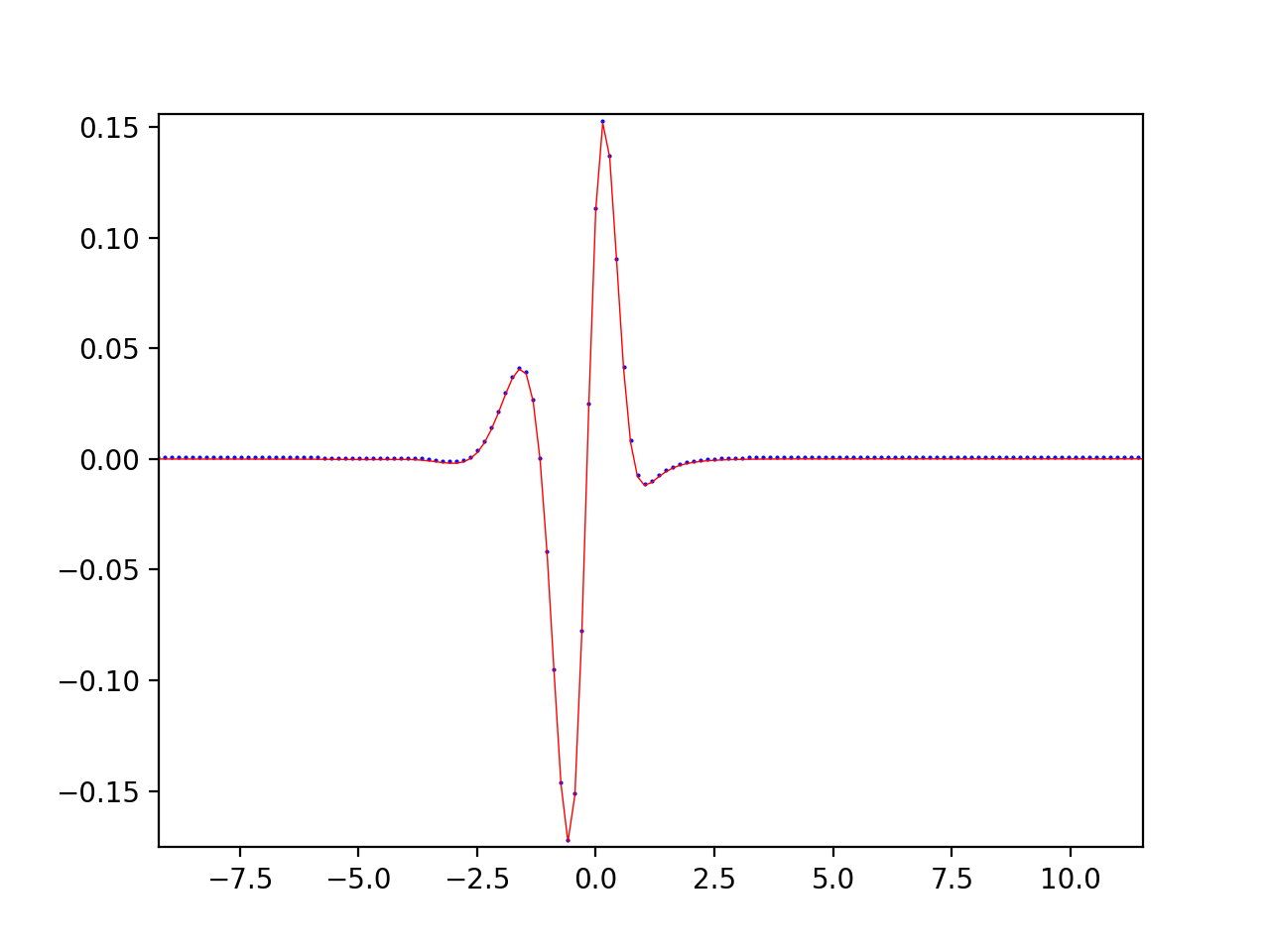}
\caption{The cross-sectional profile $u(x,0,t=2)$ of the exact solution \eqref{exact_sol}, in the linear regime ($\alpha =0$), compared with the numerical solution obtained using the scheme given in \eqref{two_step_scheme}. The modeling parameters are $\gamma = 1$,  $a = 1$. The numerical solution is shown as the pointed line, while the exact solution is displayed by the solid line. }
\label{lineargammapositive}
\end{figure}

\begin{figure}[ht]
\centering
 \includegraphics[width=0.7\textwidth, height=0.4\textwidth ]{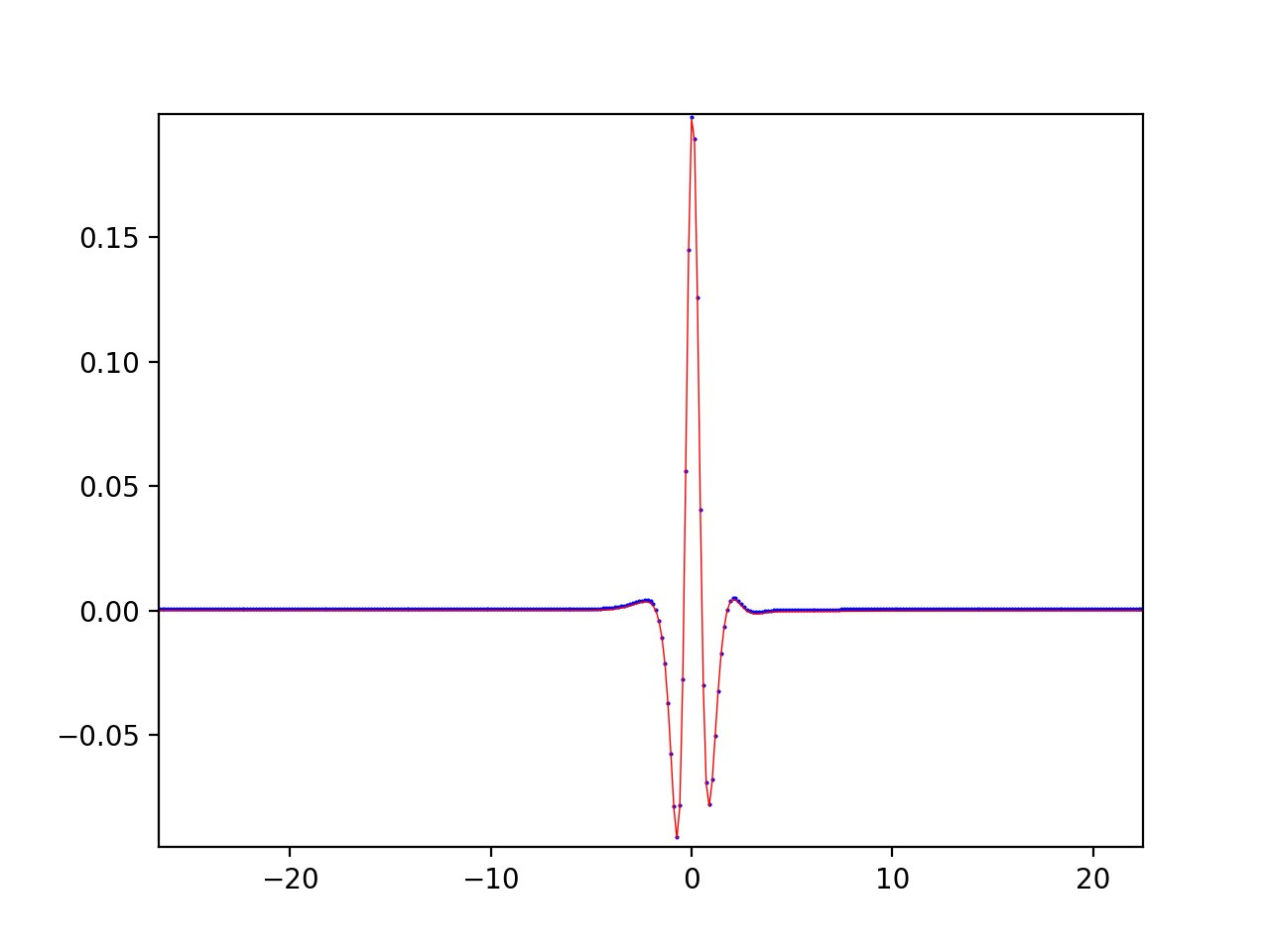}
\caption{The cross-sectional profile $u(x,0,t=2)$ of the exact solution \eqref{exact_sol}, in the linear regime ($\alpha = 0$), compared with the numerical solution. The modeling parameters are $\gamma = -1$, $a=1$. The numerical solution is shown as the pointed line, while the exact solution is displayed by the solid line.}
\label{lineargammanegative}
\end{figure}

In the next numerical simulation, we apply the solver given by \eqref{two_step_scheme} to the periodic traveling wave solution \eqref{periodic_wave_BBMKP}. The results are presented for different sets of parameters at time $t=2$.
For $\gamma = -1$, $a = 0$ (without damping), $r_1 = -0.687677$, $r_2=0.729637$, $h = -1$, $r_3=11.958$, $c = 1$, $r=2$, $\alpha=1$, the outcome is shown in Figure \ref{surf_solitarywave2}. For
$\gamma = 1$, $a=0$, $r_1 = -1.09808$, $r_2=1.5$, $r_3= 4.09808$, $h = -9/4$, $c = -1$, $r=1$, $\alpha=-2$, the result is shown in Figure \ref{surf_solitarywave3}. In these experiments, the numerical parameters are as follows: $\Delta t = 0.001$, $\Delta x = 4.9185/2^9 \approx 0.0096$, $\Delta y =2.4593/2^9 \approx 0.0048$, $\gamma = -1$; and
$\Delta x = \Delta y = 3.9847/2^9 \approx 0.0078$, for $\gamma = -1$. The maximum norm error between the exact traveling wave solution \eqref{periodic_wave_BBMKP} and the numerical solution is approximately $5e-5$ for $\gamma = -1$ and $7e-6$ for $\gamma =1$, indicating an excellent agreement between theoretical and numerical results between different values of parameters.

Finally, we use the solver \eqref{two_step_scheme}
to investigate the exponential decay of solutions to the BBM-KP equation \eqref{BBMKP2}, as established in previous sections.
Figures \ref{decay1gamma} and \ref{decay2gamma} show the logarithm of the $L^2$-norm $\|u\|_{L^2}$ as a function of time $t$, obtained for $\gamma = 1$ and $\gamma=-1$, respectively. The simulations are conducted with $\alpha = 1$ (nonlinearity), and $a = 1$ (damping), and the initial condition is a Gaussian pulse of the form \eqref{Gaussian_pulse} with $\sigma=4$.

Observe that the plots exhibit a linear trend, consistent with the exponential decay predicted by the theory discussed in Section 3. The numerical simulations use a time step of $\Delta t = 0.001$ and spatial discretizations of $\Delta x = \Delta y = 600/2^{12} \approx 0.1465$.

We repeat the previous numerical simulations with the same modeling parameters, but with $p=2$ in the BBM-KP equation \eqref{BBMKP2}. The results, presented in Figures \ref{decay3gamma}, \ref{decay4gamma}, reveal an exponential decay. However, this observation does not have a theoretical proof at present. This behavior has only been established theoretically in this paper for the case where $p=1$.

\begin{figure}[ht]
\centering
 \includegraphics[width=1.1\textwidth, height=0.8\textwidth ]{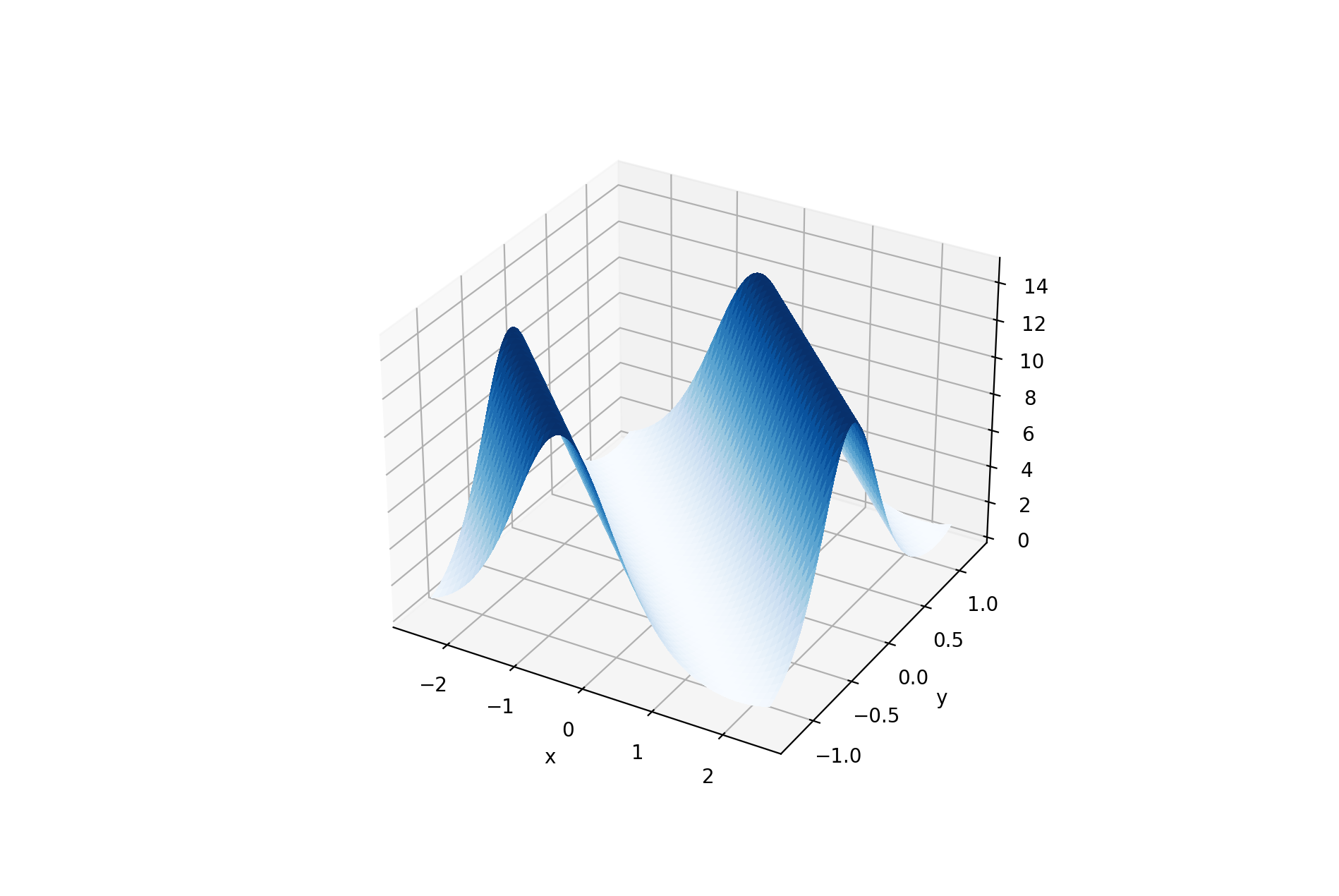}
\caption{Periodic travelling wave solution \eqref{periodic_wave_BBMKP} for $\gamma = 1$, at $t = 2$.}
\label{surf_solitarywave2}
\end{figure}

\begin{figure}[ht]
\centering
 \includegraphics[width=1.1\textwidth, height=0.8\textwidth ]{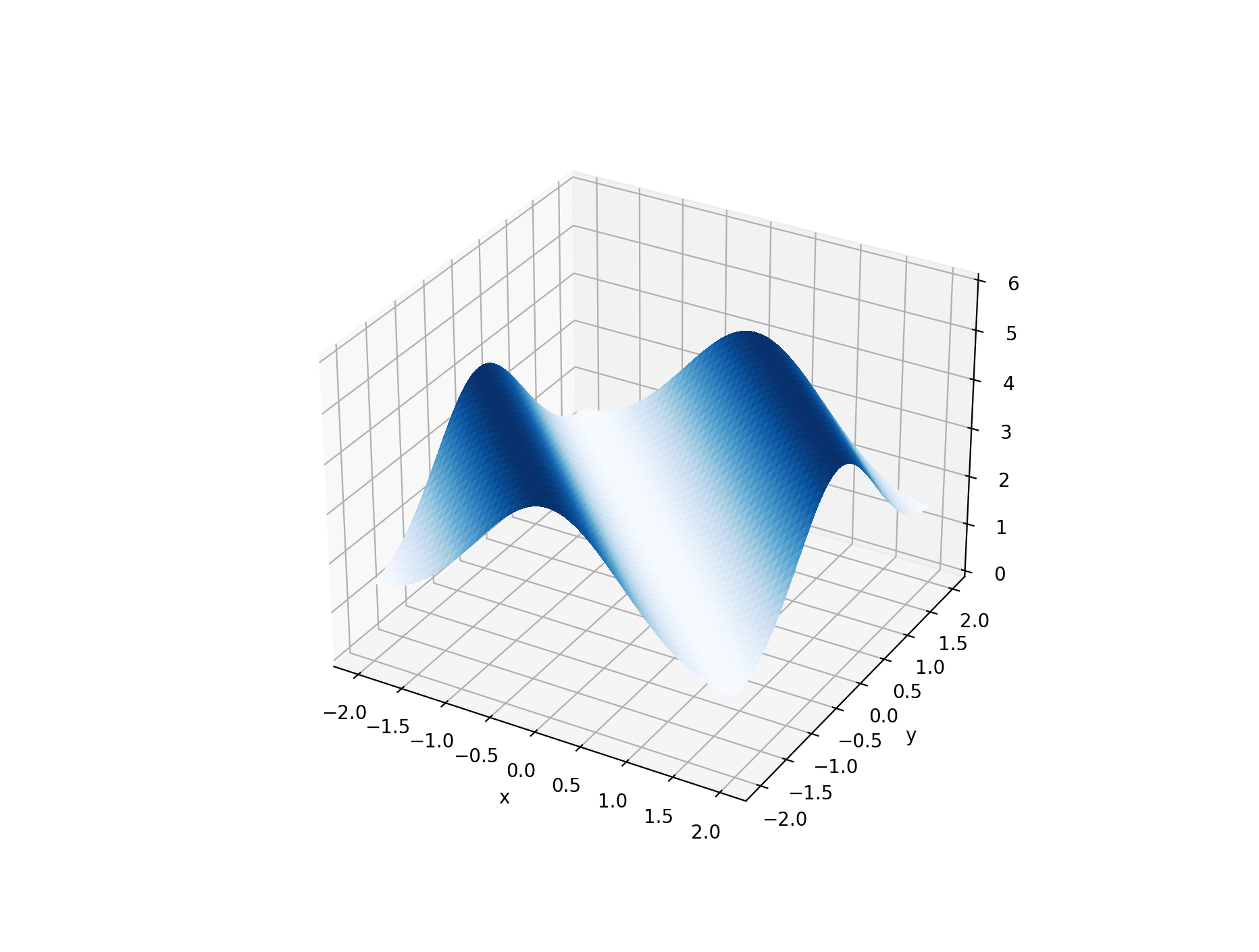}
\caption{Periodic travelling wave solution \eqref{periodic_wave_BBMKP} for $\gamma=-1$, at $t = 2$.}
\label{surf_solitarywave3}
\end{figure}

\begin{figure}[ht]
\centering
 \includegraphics[width=0.8\textwidth, height=0.4\textwidth ]{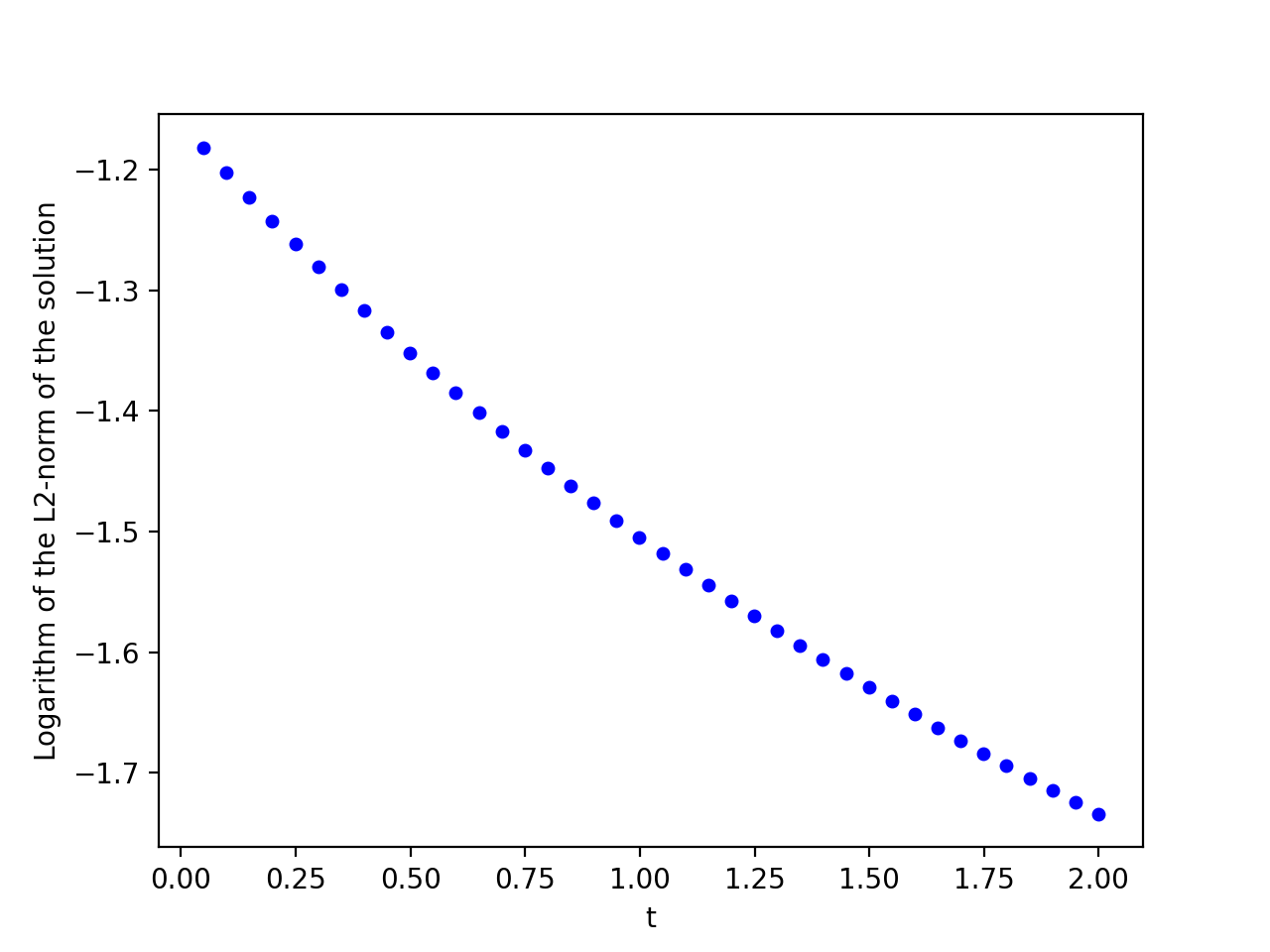}
\caption{Logarithm of the $L^2$-norm $\|u\|_{L^2}$ of a solution $u$ of equation \eqref{BBMKP2}, as a function of time $t$ for $\gamma=1, p=1$. }
\label{decay1gamma}
\end{figure}

\begin{figure}[ht]
\centering
 \includegraphics[width=0.8\textwidth, height=0.4\textwidth ]{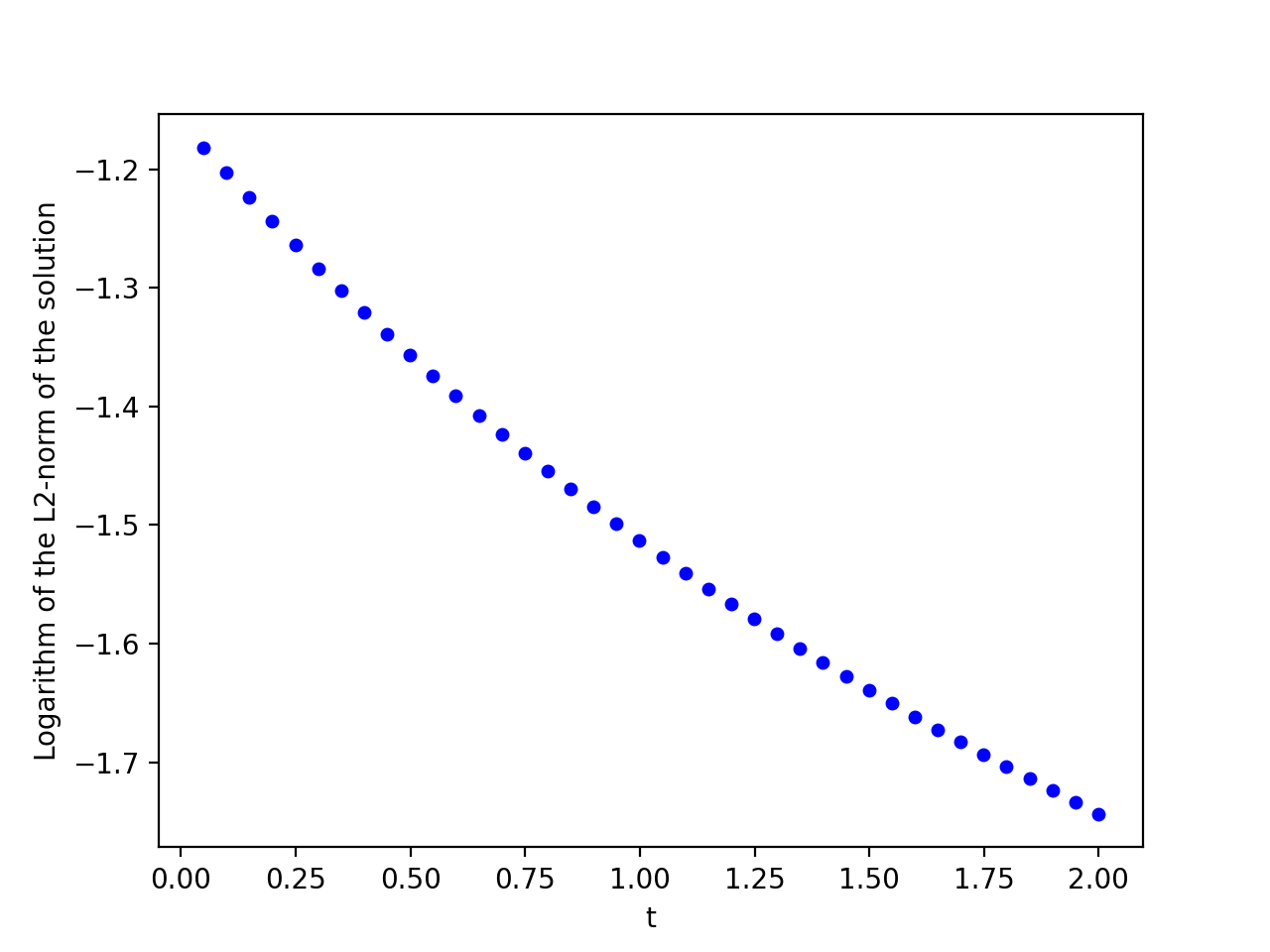}
\caption{Logarithm of the $L^2$-norm $\|u\|_{L^2}$ of a solution $u$ of equation \eqref{BBMKP2}, as a function of time $t$ for $\gamma = -1, p=1$. }
\label{decay2gamma}
\end{figure}

\begin{figure}[ht]
\centering
 \includegraphics[width=0.8\textwidth, height=0.4\textwidth ]{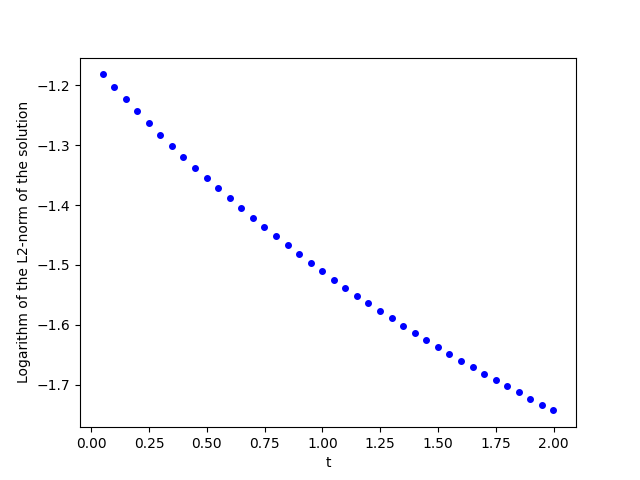}
\caption{Logarithm of the $L^2$-norm $\|u\|_{L^2}$ of a solution $u$ of equation \eqref{BBMKP2}, as a function of time $t$ for $\gamma=1, p=2$. }
\label{decay3gamma}
\end{figure}

\begin{figure}[ht]
\centering
 \includegraphics[width=0.8\textwidth, height=0.4\textwidth ]{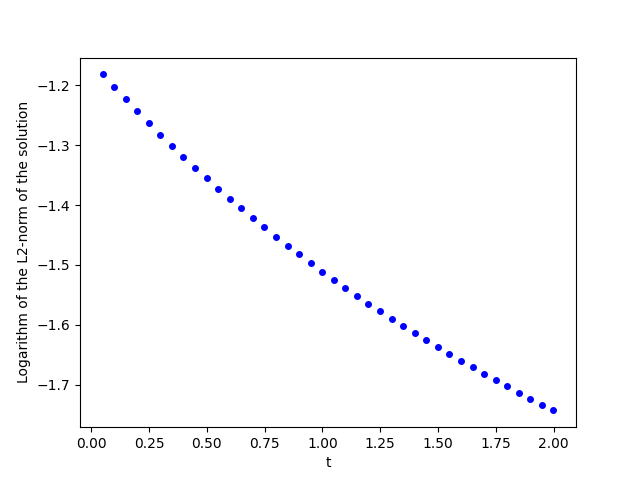}
\caption{Logarithm of the $L^2$-norm $\|u\|_{L^2}$ of a solution $u$ of equation \eqref{BBMKP2}, as a function of time $t$ for $\gamma = -1, p=2$. }
\label{decay4gamma}
\end{figure}

\section{Further Comments}\label{furthercomments}

This work addresses stabilizing the BBM-KP equation in an unbounded domain, incorporating a damping term localized in both the \(x\) and \(y\) directions. We establish that the energy associated with this equation decays exponentially to zero over time. The core of this analysis lies in demonstrating a unique continuation property, which comes from energy dissipation methods and an observability inequality derived from the compactness-uniqueness approach.

One of the significant results of this work is that the exponential decay rate is uniform and does not depend on the size of the initial data. This implies that the system exhibits global uniform exponential stabilization, ensuring robust long-term stability under a wide range of initial conditions. However, it is important to note that our results hold under certain regularity assumptions. Specifically, we restrict the initial data, requiring it to belong to a Sobolev space with regularity \(s > \frac{3}{2}\),  this condition stems from the global well-posedness. In summary, while we achieve strong stabilization results, the dependence on the regularity of the initial data presents a notable limitation, which could be an interesting direction for future research. Further investigations may focus on relaxing these regularity constraints or extending the analysis to other variations of the BBM-KP equation.

We also conducted numerical simulations using a spectral solver in conjunction with a second-order time integration scheme for the BBM-KP equation. The precision of the proposed numerical method was evaluated by comparing its results with some exact solutions of the BBM-KP equation. Additionally, we numerically validated the exponential time decay of these solutions caused by the damping term, testing various values of the model parameters.

\vglue0.4cm
\noindent Finally, let us pose some natural questions or open problems:\\

\begin{enumerate}
    \item[$(\mathcal{A})$] {\em Less regular initial data:} A significant natural open problem lies in relaxing the regularity requirements for the initial data. In particular, while \cite{saut2004} demonstrates the global well-posedness of the BBM-KP equation in the energy space, obtaining the corresponding observability inequality remains a challenge due to the lack of compactness propagation in this setting. Further investigations may focus on relaxing these regularity constraints or extending the analysis to other variations of the BBM-KP equation. \\
    
    \item[$(\mathcal{B})$] {\em Another nonlinear term $u^pu_x$ with $p=2$:} The principal issue for $p = 2$ lies in the unique continuation property (UCP). \textcolor{black}{Using} the method presented in Section \ref{UCP}, it becomes clear that the conclusion is based on the fact that $p+1$ is an even number. This observation leads to the possibility of extending the same UCP to the generalized BBM-KP equation when $p$ is an odd integer. In such cases, the UCP holds because of similar structural properties of the equation. However, for odd values of $p$, a major challenge arises in terms of well-posedness, as noted in Remark \ref{remark1}. In contrast, when $p$ is even, we achieve global well-posedness, as demonstrated in the case of $p = 2$. Our numerical experiments suggest that exponential decay is also observed for the nonlinear exponent $p=2$, a case not covered in the current paper. This remains as an open problem to be addressed in future research.\\ 
    
    \item[$(\mathcal{C})$] {\em Damping Source localized in one direction:} The unique continuation property (UCP) for the BBM-KP equation is achieved when the solution is supported in a square region like \([ -B, B ] \times [ C, D ]\). In such a setting, the stabilization problem can be effectively reduced to proving the UCP, provided the damping term is localized within the same square region. However, a more complex situation arises when the damping term is localized in different geometric configurations. For example, if the damping is supported in a strip such as \(\mathbb{R} \times [ C, D ]\) or \([ -B, B ] \times \mathbb{R}\), the UCP must be established for solutions \(u\) that are supported within these unbounded strip-like regions. The challenge here stems from the difficulty in controlling the propagation of information or energy dissipation in unbounded domains, which complicates the application of classical UCP techniques. This presents a natural and significant open problem in our context. Achieving the UCP in these geometries is crucial for extending the stabilization results to more general damping configurations. 

 \end{enumerate}

\subsection*{Acknowledgment} 
This work began while the first author was visiting the Universidade Federal de Pernambuco and was completed when the second author was visiting the Universidad Nacional de Colombia Sede Manizales.  Gonzalez Martinez was supported by CAPES/COFECUB grant 88887.879175/2023-00, CNPq grant 421573/2023-6 and Propesqi. The third author was partially supported by the Universidad del Valle, under research project C.I. 71360. The authors thank the host institutions for their warm hospitality.


\begin{thebibliography}{99}

\bibitem{alam2013}
M. N. Alam and M. A. Akbar, Exact traveling wave solutions of the KP-BBM equation by using the new approach of generalized (G’/G)-expansion method, SpringerPlus, 2 (2013), p. 617.
\bibitem{aguilar2024}
J.B Aguilar, and Tom, M. M, Convergence of solutions of the BBM and BBM-KP model equations. Differential and Integral Equations, 37(3/4),  (2024), 187-206.
\bibitem{bona2002}
J. Bona, Y. Liu, and M. M. Tom, The Cauchy problem and stability of solitary-wave
solutions for RLW-KP-type equations, J. Differential Equations, 185 (2002), 437-482.
\bibitem{capistrano2024}
R.A. Capistrano,  F.A. Gallego and V. Komornik.  A qualitative study of the generalized dispersive systems with time-delay: The unbounded case. (preprint)
\bibitem{cavalcanti2012}
M.M. Cavalcanti, V.N. Domingos Cavalcanti, A. Faminskii, and F. Natali, Decay of solutions to damped Korteweg–de Vries type equation. Applied Mathematics and Optimization, 65, (2012) 221-251.
\bibitem{cavalcanti2014}
M. M. Cavalcanti, V. N. Domingos Cavalcanti, V. Komornik, and J. H. Rodrigues, Global well-posedness and
exponential decay rates for a KdV–Burgers equation with indefinite damping, Ann. I. H. Poincar´e-AN, 31: 1079–
1100 (2014).
\bibitem{constantin}
A. Constantin, Finite propagation speed for the Camassa-Holm equation, J. Math.
Phys., 46 (2005), 023506, 4.
\bibitem{daspan2007}
G.P. Daspan and M. Tom, Comparison of KP and BBM-KP models. Int. J. Math. Math. Sci.(2007), Art. ID 37274, 20 pp.
\bibitem{doronin2016}
G.G. Doronin, F. Natali, Exponential decay for a locally damped fifth-order equation posed on the line, Nonlinear Anal., Real World Appl. 30 (2016) 59–72.
\bibitem{faminskii2021}
A. V. Faminskii and E. V. Martynov, Large-time decay of solutions of the damped Kawahara equation on the half-line, Differential Equations on Manifolds and Mathematical Physics, V. M. Manuilov et al. eds. Trends in Mathematics, Birkhauser, 2021, 130–141.
\bibitem{gallego2019}
F.A. Gallego and A.F. Pazoto.  On the well-posedness and asymptotic behavior of the generalized KdV-Burgers equation.  Proc. R. Soc. Edinb. A: Math. (2019) , Vol 149, Issue 1, Pages 219 - 260.
\bibitem{komornik2020}
V. Komornik and C. Pignotti, Well-posedness and exponential decay estimates for a Korteweg–de Vries–Burgers
equation with time-delay, Nonlinear Analysis (191), 111646 (2020).
\bibitem{linares2009}
Linares, F., Pazoto, A.F.: Asymptotic behavior of the Korteweg–de Vries equation posed in a quarter
plane. J. Differ. Equ. 246, 1342–1353 (2009)
\bibitem{littman}
Littman, W.: Near optimal time boundary controllability for a class of hyperbolic equations. In: Control
Problems for Systems Described by PDE’s and Applications. Lecture Notes in Control and Information
Sciences, vol. 97, pp. 307–313. Springer, Berlin (1987)
\bibitem{lions}
J.L. Lions, Controlabilite Exacte, Perturbationset Stabilisation de Systemes Distribues. Vol. 22. Elsevier-Masson, 1990
\bibitem{mammeri}
Y. Mammeri, Unique Continuation Property for the KP-BBM-II Equation, Differential and Integral Equations, Volume 22, Numbers 3-4 (2009), 393-399
\bibitem{menzala2002}
G.P. Menzala, C.F. Vasconcellos, E. Zuazua, Stabilization of the Korteweg–de Vries equation with localized damping,
Quart. Appl. Math. 60 (2002) 111–129.
\bibitem{ouyang2014}
Z. Ouyang, Traveling wave solutions of the kadomtsev-petviashvili-benjamin-bona-mahony
equation, in Abstract and Applied Analysis, vol. 2014, Hindawi, 2014.
\bibitem{pazoto2010}
A.F. Pazoto, L. Rosier, Uniform stabilization in weighted Sobolev spaces for the KdV equation
posed on the half-line. Discrete Contin. Dyn. Syst., Ser. B 14, (2010), 1511–1535.
\bibitem{rosier}\textcolor{black}{
L. Rosier, B.Y  Zhang,  Unique continuation property and control for the Benjamin–Bona–Mahony equation on a periodic domain. Journal of Differential Equations, 254 (1), (2013), 141–178.}
\bibitem{saut2004}
J.-C. Saut and N. Tzvetkov, Global well-posedness for the KP-BBM equations, AMRX
Appl. Math. Res. Express, 1 (2004), 1-16.
\bibitem{saut1993_2}
J.-C. Saut, On the Cauchy problem for the Kadomtsev-Petviashvili equation, Geom. Funct. Anal. 3
(1993), no. 4, 315–341.
\bibitem{saut2001} 
J.-C. Saut and N. Tzvetkov, On periodic KP-I type equations, Comm. Math. Phys. 221 (2001),
no. 3, 451–476.
\bibitem{simon1986compact}
J. Simon, Compact sets in the space $L^p(0; T;B)$, Ann. Mat. Pura Appl. 146 (1986), 65-96.
\bibitem{song2010}
M. Song, C. Yang, and B. Zhang, Exact solitary wave solutions of the Kadomtsov–
Petviashvili–Benjamin–Bona–Mahony equation, Applied Mathematics and Computation, 217 (2010), pp. 1334–1339.
\bibitem{wang2021}
Wang, M. and Zhou, D.. Exponential decay for the linear KdV with a rough localized damping. Appl. Math. Lett. 120 (2021), 107264
\bibitem{wang2023}
Wang M, Zhou D. Exponential decay for the KdV equation on $\R$ with new localized dampings. Proceedings of the Royal Society of Edinburgh: Section A Mathematics. 2023, 153(4):1073-1098.
\bibitem{wazwaz2005}
A.-M. Wazwaz, Exact solutions of compact and noncompact structures for the KP–BBM equation, Applied Mathematics and Computation, 169 (2005), pp. 700–712.
\bibitem{wazwaz2008}
A.-M. Wazwaz, The extended tanh method for new compact and noncompact solutions for the KP–BBM
and the ZK–BBM equations, Chaos, Solitons \& Fractals, 38 (2008), pp. 1505–1516.
\bibitem{Zuily} C. Zuily, Éléments de distributions et d'équations aux dérivées partielles: cours et problèmes résolus. Vol. 130. Dunod, 2002.
\end{thebibliography}
\end{document}